\documentclass{amsart}
\usepackage{amsmath,amsthm,amsfonts,amssymb,amscd,mathrsfs}
\usepackage{psfrag,graphicx}

\usepackage{epic,eepic}
\usepackage{color}
\usepackage{ebezier}



\theoremstyle{plain}

\newtheorem{Thm}{Theorem}[section]
\newtheorem{Lem}[Thm]{Lemma}
\newtheorem{Prop}[Thm]{Proposition}

\theoremstyle{remark}
\newtheorem{Def}[Thm] {Definition}
\newtheorem{Rem}[Thm] {Remark}

\long\def\begcom#1\endcom{}

\newcommand{\Leb}{\operatorname{Leb}}
\newcommand{\diam}{\operatorname{diam}}
\newcommand{\orb}{\operatorname{orb}}
\newcommand{\length}{\operatorname{\length}}

\newcommand{\dist}{\operatorname{dist}}

\newcommand{\Int}{\operatorname{int}}

\def\supp{\operatorname{supp}}

\def\id{\operatorname{id}}
\def\length{\operatorname{length}}
\def\loc{\operatorname{loc}}

\def\top{\operatorname{top}}

\begin{document}

\title[ Saturated Sets for Nonuniformly  Hyperbolic Systems  ]
      { Saturated Sets for Nonuniformly  Hyperbolic Systems}

\author[Liang, Liao, Sun, Tian]
{Chao Liang$^{*}$, Gang Liao$^{\dag}$, Wenxiang Sun$^{\dag}$,
Xueting Tian$^{\ddag}$}

\thanks{$^{*}$ Applied Mathematical Department, The Central University of Finance and Economics,
Beijing 100081, China; Liang is supported by NNSFC(\# 10901167)}\email{chaol@cufe.edu.cn}

 \thanks{$^{\dag}$ School of Mathematical Sciences,
Peking University, Beijing 100871,
China; Sun is supported by NNSFC(\#
10231020) and Doctoral Education Foundation of China}\email{liaogang@math.pku.edu.cn}
\email{sunwx@math.pku.edu.cn}

 \thanks{$^{\ddag}$
Academy of Mathematics and Systems Science, Chinese Academy of
Sciences, Beijing 100190, China; Tian is supported by CAPES} \email{tianxt@amss.ac.cn}

\date{July, 2011}

\maketitle


\begin{abstract}
In this paper we prove that for  an ergodic hyperbolic measure
$\omega$ of a $C^{1+\alpha}$ diffeomorphism $f$ on a Riemannian
manifold $M$, there is an $\omega$-full measured set
$\widetilde{\Lambda}$ such that for every invariant probability
$\mu\in \mathcal{M}_{inv}(\widetilde{\Lambda},f)$, the metric
entropy of $\mu$ is equal to the topological entropy of saturated
set $G_{\mu}$ consisting of  generic  points of $\mu$:
$$h_\mu(f)=h_{\top}(f,G_{\mu}).$$
Moreover, for every nonempty, compact and connected subset $K$ of
$\mathcal{M}_{inv}(\widetilde{\Lambda},f)$ with the same hyperbolic
rate, we compute the topological entropy of saturated set $G_K$ of
$K$ by the following equality:
$$\inf\{h_\mu(f)\mid \mu\in K\}=h_{\top}(f,G_K).$$

In particular these results can be applied (i) to the nonuniformy
hyperbolic diffeomorphisms  described by Katok, (ii) to the robustly
transitive partially hyperbolic diffeomorphisms described by
~Ma{\~{n}}{\'{e}}, (iii) to the robustly transitive non-partially
hyperbolic diffeomorphisms described  by Bonatti-Viana. In all these
cases $\mathcal{M}_{inv}(\widetilde{\Lambda},f)$ contains an open
subset of $\mathcal{M}_{erg}(M,f)$.
\end{abstract}

\tableofcontents

\section{Introduction}

Let $(M, d)$ be  a compact metric space and $f : M\rightarrow M$ be
a continuous map. Given an invariant  subset $\Gamma\subset M$,
denote  by  $\mathcal{M}(\Gamma)$ the set consisting of all Borel
probability  measures,  by $\mathcal{M}_{inv}(\Gamma, f)$  the
subset consisting of $f$-invariant probability measures and, by
$\mathcal{M}_{erg}(\Gamma, f)$ the subset consisting of
$f$-invariant ergodic probability measures. Clearly, if $\Gamma$ is
compact then $\mathcal{M}(\Gamma)$ and $\mathcal{M}_{inv}(\Gamma,
f)$ are both compact spaces in the weak$^*$-topology of measures.
Given $x\in M$, define the $n$-ordered empirical measure
 of $x$ by
$$\mathcal{E}_n(x)=\frac1n\sum_{i=0}^{n-1}\delta_{f^i(x)},$$
where $\delta_y$ is the Dirac mass at $y\in M$. A subset $W\subset
M$ is called saturated if $x\in W$ and the sequence
$\{\mathcal{E}_n(y)\}$ has the same limit points set as
$\{\mathcal{E}_n(x)\}$ then $y\in W$. The limit point set $V(x)$ of
$\{\mathcal{E}_n(x)\}$ is always a compact connected subset  of $
\mathcal{M}_{inv}(M, f)$. Given $\mu\in \mathcal{M}_{inv}(M, f)$, we
collect the saturated set $G_\mu$ of $\mu$ by those generic points
$x$ satisfying $V(x)=\{\mu\}$. More generically, for a compact
connected subset $K\subset \mathcal{M}_{inv}(M, f)$, we denote by
$G_K$ the saturated set consisting  of points $x$ with $V(x)=K$. By
Birkhoff Ergodic Theorem, $\mu(G_{\mu})=1$ when $\mu$ is ergodic.
  However, this is  somewhat a special case.  For non-ergodic $\mu$, by Ergodic Decomposition Theorem,  $G_\mu$ has measure $0$ and
thus is `` thin " in view of measure.  In addition, when $f$ is
uniformly hyperbolic (\cite{Sigmund}) or non uniformly hyperbolic
(\cite{LST}), $G_\mu$ is of first category hence ``thin" in view of
topology.  Exactly, one can  get this topological fact of first
category as follows. Denote by $C^0(M)$ the set of continuous
real-valued functions on $M$ provided with the sup norm. For non
uniformly hyperbolic systems $(f,\mu)$,  there is $x\in M$ such that
$$\overline{\orb(x,f)}\supset \supp(\mu)\quad \mbox{and}\quad
\mathcal{E}_n(x)\quad \mbox{does not converge},$$ where the support
of a measure $\nu$, denoted by  $\supp(\nu)$, is the minimal closed
set with $\nu-$total measure, see \cite{Sigmund,LST}.  We can take
$0<a_1<a_2$ and $ \varphi\in C(M)$ such that
$$\liminf_{n\rightarrow +\infty} \frac1n\sum_{i=0}^{n-1}\varphi(f^i(x))<a_1<a_2<\limsup_{n\rightarrow +\infty}\frac1n\sum_{i=0}^{n-1}\varphi(f^i(x)).$$
Let $$R=\cap_N \cup_{n\geq N}\big{\{}x\mid
\frac1n\sum_{i=0}^{n-1}\varphi(f^i(x))<a_1\big{\}}\cap \cap_N
\cup_{n\geq N}\big{\{}x\mid
\frac1n\sum_{i=0}^{n-1}\varphi(f^i(x))>a_2\big{\}}.$$ Then $$R\cap
\overline{\orb(x,f)} \subset (\overline{\orb(x,f)}\setminus
G_{\mu})\,\,\mbox{ and} \,\, R\cap \overline{\orb(x,f)}\,\,\mbox{ is
a}\,\, G_{\delta}\,\,\mbox{ subset of}\,\, \overline{\orb(x,f)}.$$
Combining with $x\in \overline{\orb(x,f)}$, we can see that
$\overline{\orb(x,f)}\setminus G_{\mu}$ is a residual set of
$\overline{\orb(x,f)}$. Hence, $G_{\mu}$ is of first category in the
subspace $\overline{\orb(x,f)}$.

For a conservative system $(f,M,\Leb)$ preserving the normalized
volume measure $\Leb$, if $f$ is ergodic, then by ergodic theorem,
$$\mathcal{E}_n(x)\rightarrow \Leb,\quad \mbox{as}\quad n\rightarrow +\infty,$$
for $\Leb$-a.e. $x\in M$. In the general dissipative case where, a
priori, there is no distinguished invariant probability measures, it
is much more subtle what one should mean by describing the behavior
of  almost orbits in the physically  observable sense. In this
content, an invariant measure $\mu$ is called physical measure (or
Sinai-Ruelle-Bowen meaure) if the saturated set $G_{\mu}$ is of
positive Lebesgue measure. SRB measures are used to measure the
``thickness'' of saturated sets in view of $\Leb$-measure.

  Motivated by the definition of
saturated sets, it is reasonable to think that $G_{\mu}$ should put
together all information of
 $\mu$. If $\mu$ is ergodic,
Bowen\cite{Bowen4} has succeeded this motivation to prove that
$$h_{\top}(f,G_{\mu})=h_\mu(f).$$
When $f$ is mixing and uniformly hyperbolic (which implies uniform
specification property), applying \cite{PS} it also holds that
$$h_{\top}(f,G_{\mu})=h_\mu(f).$$ This implies that $G_\mu$ is
``thick" in view of topological entropy.  Indeed, the information of
invariant measure can be well approximated by nearby measures
\cite{Katok2, Wang, Liang,Liao-Sun-Tian}.  For non uniformly
hyperbolic systems, in \cite{LST} Liang, Sun and Tian proved
$G_{\mu}\neq \emptyset$. Our goal in the present paper is to show
the ``thickness" of $G_{\mu}$ in view of entropy.

Now we start to introduce our results precisely. Let $M$ be a
compact connected boundary-less Riemannian $d$-dimensional manifold
and $f : M \rightarrow M$ a $C^{1+\alpha}$ diffeomorphism.
 We
use $Df_x$ to denote the tangent map of $f$ at $x\in M$. We say that
$x\in M$ is a regular point of $f$ if there exist numbers
$\lambda_1(x)>\lambda_2(x)>\cdots>\lambda_{\phi(x)}(x)$ and a
decomposition on the tangent space
$$T_xM=E_1(x)\oplus\cdots\oplus E_{\phi(x)}(x)$$
such that$$\underset{n\rightarrow
\infty}{\lim}\frac{1}{n}\log\|(D_xf^n)u\|=\lambda_j(x)$$ for every
$0\neq u\in E_j(x)$ and every $1\leq j\leq \phi(x)$. The number
$\lambda_i(x)$ and the space $E_i(x)$ are called the Lyapunov
exponents and the eigenspaces of $f$ at the regular point $x$,
respectively.  Oseledets theorem \cite{Oseledec} states that all
regular points of a diffeomorphism $f: M\rightarrow M$ forms a Borel
set with total measure.  For a regular point $x\in M$ we define
$$\lambda^+(x)=\max\{0,\,\min\{\lambda_i(x)\mid
\lambda_i(x)>0,\,1\leq i\leq \phi(x)\}\} $$ and
$$\lambda^-(x)=\max\{0,\,\min\{-\lambda_i(x)\mid \lambda_i(x)<0,\,1\leq i\leq
\phi(x)\}\},
$$ We appoint $\min \emptyset=\max \emptyset =0$. Taking an ergodic invariant measure $\mu$, by the ergodicity for
$\mu$-almost all $x\in M$ we can obtain uniform exponents
$\lambda_i(x)=\lambda_i(\mu)$ for $1\leq i\leq \phi(\mu)$. In this
content we denote $\lambda^+(\mu)=\lambda^+(x)$ and
$\lambda^-(\mu)=\lambda^-(x)$. We say an ergodic measure $\mu$ is
hyperbolic if $\lambda^+(\mu)$ and $\lambda^-(\mu)$ are both
non-zero.

\begin{Def}\label{Def6} Given $\beta_1,\beta_2\gg\epsilon
>0$, and for all $k\in \mathbb{Z}^+$,  the hyperbolic block
$\Lambda_k=\Lambda_k(\beta_1,\beta_2;\,\epsilon)$ consists of all
points $x\in M$ for which there is a splitting $T_xM=E_x^s\oplus
E_x^u$ with the invariance property $Df^t(E_x^s)=E_{f^tx}^s$ and
$Df^t(E_x^u)=E_{f^tx}^u$, and satisfying:\\
 $(a)~
\|Df^n|E_{f^tx}^s\|\leq  e^{\epsilon k}e^{-(\beta_1-\epsilon)
n}e^{\epsilon|t|}, \forall t\in\mathbb{Z}, n\geq1;$\\
 $(b)~
\|Df^{-n}|E_{f^tx}^u\|\leq  e^{\epsilon k}e^{-(\beta_2-\epsilon)
n}e^{\epsilon|t|}, \forall t\in\mathbb{Z}, n\geq1;$ and\\
 $(c)~
\tan(Angle(E_{f^tx}^s,E_{f^tx}^u))\geq e^{-\epsilon
k}e^{-\epsilon|t|}, \forall t\in\mathbb{Z}.$
\end{Def}

\begin{Def}\label{Def7} $\Lambda(\beta_1,\beta_2;\epsilon)=\overset{+\infty}{\underset{k=1}{\cup}}
\Lambda_k(\beta_1,\beta_2;\epsilon)$ is a Pesin set.
\end{Def}

It is verified that $\Lambda(\beta_1,\beta_2;\epsilon)$ is an
$f$-invariant set but usually not compact. Although the definition
of Pesin set is adopted in a topology sense, it is indeed related to
invariant measures.  Actually,  given an ergodic hyperbolic measure
$\omega$ for $f$ if $\lambda^-(\omega)\geq  \beta_1$ and
$\lambda^+(\mu)\geq\beta_2$ then $\omega\in
\mathcal{M}_{inv}(\Lambda(\beta_1,\beta_2;\epsilon), f)$. From now
on we fix such a measure $\omega$ and  denote by
$\omega\mid_{\Lambda_l}$ the conditional measure of ¥ø$\omega$ on
$\Lambda_l$. Set
$\widetilde{\Lambda}_l=\supp(\omega\mid_{\Lambda_l})$ and
$\widetilde{\Lambda}=\cup_{l\geq1}\widetilde{\Lambda}_l$.   Clearly,
$f^{\pm}(\widetilde{\Lambda}_l)\subset \widetilde{\Lambda}_{l+1}$,
and the sub-bundles $E^s_x$, $E^u_x$ depend continuously on $x\in
\widetilde{\Lambda}_l$. Moreover, $\widetilde{\Lambda}$ is also
$f$-invariant with $\omega$-full measure \footnote{Here
$\widetilde{\Lambda}$ is obtained by taking support for each
hyperbolic block $\Lambda_l$ so even if an ergodic measure with
Lyapunov exponents away from $[-\beta_1,\beta_2]$ it is not
necessary of positive measure for $\widetilde{\Lambda}$. We will
give more discussions on $\widetilde{\Lambda}$ in section 6}.

\begin{Thm}\label{main theorem of measure} For every $\mu\in \mathcal{M}_{inv}(\widetilde{\Lambda},
f)$, we have
$$h_{\mu}(f)=h_{\top}(f,G_{\mu}).$$
\end{Thm}

Let $\{\eta_l\}_{l=1}^{\infty}$ be a decreasing sequence which
approaches zero. As in \cite{Newhouse} we say a probability measure
$\mu\in \mathcal{M}_{inv}(M, f)$ has {\it hyperbolic rate}
$\{\eta_l\}$ with respect to the Pesin set
$\widetilde{\Lambda}=\cup_{l\geq 1}\widetilde{\Lambda}_l$ if
$\mu(\widetilde{\Lambda}_l)\geq 1-\eta _l$ for all $l\geq1$.
\begin{Thm}\label{main theorem of set}Let $\eta=\{\eta_l\}$ be a sequence decreasing to zero and
$\mathcal{M}(\widetilde{\Lambda},\eta)\subset\mathcal{M}_{inv}(M,f)
$ be the set of measures with hyperbolic rate $\eta$. Given any
nonempty compact connect   set $K\subset
\mathcal{M}(\widetilde{\Lambda},\eta)$, we have
$$\inf\{h_{\mu}(f)\mid \mu\in K\}=h_{\top}(f,G_K).$$

\end{Thm}

\section{Dynamics of non uniformly hyperbolic systems}\setlength{\parindent}{2em}
 We start with some notions and results of Pesin theory \cite{Barr-Pesin,Katok1,Pollicott}.

\subsection{Lyapunov metric}
Assume
$\Lambda(\beta_1,\beta_2;\,\epsilon)=\cup_{k\geq1}\Lambda_k(\beta_1,\beta_2;\,\epsilon)$
is a nonempty Pesin set. Let $\beta_1'=\beta_1-2\epsilon$,
$\beta_2'=\beta_2-2\epsilon$. Note that $\epsilon\ll
\beta_1,\beta_2$, then $\beta_1'>0, \beta_2'>0$.
 For $x\in \Lambda(\beta_1,\beta_2;\,\epsilon)$, we define
 $$\|v_s\|_s=\sum_{n=1}^{+\infty}e^{\beta_1'n}\|D_xf^n(v_s)\|, ~~~\forall~v_s\in E^s_x ,$$
 $$\|v_u\|_u=\sum_{n=1}^{+\infty}e^{\beta_2'n}\|D_xf^{-n}(v_u)\|, ~~~\forall~v_u\in E^u_x ,$$
$$\|v\|'=\mbox{max}(\|v_s\|_s,\,\|v_u\|_u)~~~\mbox{where}~v=v_s+v_u.$$
We call the norm $\|\cdot\|'$ a Lyapunov metric. This metric is in
general not equivalent to the Riemannian metric. With the Lyapunov
metric $f : \Lambda\rightarrow\Lambda$ is uniformly hyperbolic. The
following estimates are known :\\
 $(i) \|Df\mid_{E^s_x}\|'\leq e^{-\beta_1'},~~~ \|Df^{-1}\mid_{E^u_x}\|'\leq
 e^{-\beta_2'}$;\\
 $(ii) \frac{1}{\sqrt{2}}\|v\|_x\leq \|v\|_x'\leq \frac{2}{1-e^{-\epsilon}}e^{\epsilon k}\|v\|_x,~\forall~v\in T_x M,~x\in \Lambda_k.$
 \begin{Def}\label{Lyapunov coordinates}
In the local coordinate chart, a coordinate change $C_{\varepsilon}:
M\rightarrow GL(m,\mathbb{R})$ is called a Lyapunov change of
coordinates if for each regular point $x\in M$ and $u,v\in T_xM$, it
satisfies
$$<u,v>_x=<C_{\varepsilon}u,C_{\varepsilon}u>_x'.$$

 \end{Def}
By any Lyapunov change of coordinates $C_{\epsilon}$ sends the
orthogonal decomposition $\mathbb{R}^{\dim E^s}\oplus
\mathbb{R}^{\dim E^u} $ to the decomposition $E^s_x\oplus E^u_x$ of
$T_xM$. Additionally, denote
$A_{\epsilon}(x)=C_{\epsilon}(f(x))^{-1}Df_xC_{\epsilon}(x)$. Then
$$A_{\epsilon}(x)=\begin{pmatrix}A_{\epsilon}^s(x)&0\\0&A_{\epsilon}^u(x)\\\end{pmatrix},$$
$$\|A_{\epsilon}^s(x)\|\leq e^{-\beta_1'},\quad \|A_{\epsilon}^u(x)^{-1}\|\leq e^{-\beta_2'}.$$
\subsection{Lyapunov neighborhood}
Fix a point $x\in \Lambda(\beta_1,\beta_2;\,\epsilon)$. By taking
charts about $x, f(x)$  we can assume without loss of generality
that $x\in\mathbb{R}^d, f(x)\in \mathbb{R}^d$. For a sufficiently
small neighborhood $U$ of $x$, we can trivialize the tangent bundle
over $U$ by identifying $T_UM\equiv U\times \mathbb{R}^d$. For any
point $y\in U$ and tangent vector $v\in T_yM$ we can then use the
identification $T_UM=U\times \mathbb{R}^d$ to translate the vector
$v$ to a corresponding vector $\bar{v}\in T_xM$. We then define
$\|v\|''_y=\|\bar{v}\|'_x$, where $\|\cdot\|'$ indicates the
Lyapunov metric. This defines a new norm $\|\cdot\|''$ (which agrees
with $\|\cdot\|'$ on the fiber $T_xM$). Similarly, we can define
$\|\cdot\|''_z$ on $T_zM$ (for any $z$ in a sufficiently small
neighborhood of $fx$ or $f^{-1}x$). We write $\bar{v}$ as $v$
whenever there is no confusion. We can define a new splitting $T_yM
= {E^s_y}'\oplus {E^u_y}', y\in U$ by translating the splitting
$T_xM = E^s_x\oplus E^u_x$ (and similarly for $T_zM = {E^s_z}'\oplus
{E^u_z}'$).

There exist $\beta_1''=\beta_1-3\epsilon>0 ,
\beta_2''=\beta_2-3\epsilon>0$ and $\epsilon_0> 0 $ such that if we
set $\epsilon_k =\epsilon_0e^{-\epsilon k }$ then for any $y\in
B(x,\epsilon_k)$ in an $\epsilon_k$ neighborhood of $x\in\Lambda_k$.
We have a splitting
$T_yM = {E^s_y}'\oplus {E^u_y}'$ with hyperbolic behavior: \\
$(i)~ \|D_yf(v)\|''_{fy}\leq e^{-\beta_1''}\|v\|''¡¸$ for every
$v\in
{E^s_y}'$;\\
 $(ii)~ \|D_yf^{-1}(w)\|''_{f^{-1}y}\leq e^{-\beta_2''}\|w\|''¡¸$ for every $w\in
{E^u_y}'$.\\
The constant $\epsilon_0$ here and afterwards depends on various
global properties of $f$, e.g., the H$\ddot{o}$lder constants, the
size of the local trivialization, see p.73 in \cite{Pollicott}.
\begin{Def}\label{Def5} We define the Lyapunov neighborhood $\Pi=\Pi(x,a\epsilon_k)$
of $x\in \Lambda_k$ (with size $a\epsilon_k$, $0<a<1$) to be the
neighborhood of $x$ in $M$ which is the exponential projection onto
$M$ of the tangent rectangle
$(-a\epsilon_k,\,a\epsilon_k)E^s_x\oplus
(-a\epsilon_k,\,a\epsilon_k)E^u_x$.
\end{Def}  In the Lyapunov neighborhoods, $Df$ displays uniformly
hyperbolic in the Lyapunov metric. More precisely, one can extend
the definition of $C_{\epsilon}(x)$ to the Lyapunov neighborhood
$\Pi(x,a\epsilon_k)$ such that for any $y\in \Pi(x,a\epsilon_k)$,
$$A_{\epsilon}(y):=C_{\epsilon}(f(y))^{-1}Df_yC_{\epsilon}(y)= \begin{pmatrix}A_{\epsilon}^s(y)&0\\0&A_{\epsilon}^u(y)\\\end{pmatrix},$$
$$\|A_{\epsilon}^s(y)\|\leq e^{-\beta_1''},\quad \|A_{\epsilon}^u(y)^{-1}\|\leq e^{-\beta_2''}.$$
Let $\Psi_x=\exp_x\circ C_{\epsilon}(x)$. Given $x\in \Lambda_k$, we
say that the set $H^u\subset \Pi(x,a\epsilon_k)$ is an admissible
$(u,\gamma_0,k)$-manifold near $x$ if $H^u=\Psi_x(\mbox{graph
}\psi)$ for some $\gamma_0$-Lipschitz function $\psi:
(-a\epsilon_k,\,a\epsilon_k)E^u_x\rightarrow
(-a\epsilon_k,\,a\epsilon_k)E^s_x$ with $\|\psi\|\leq
a\epsilon_k/4$. Similarly, we can also define
$(s,\gamma_0,k)$-manifold near $x$. Through each point $y\in
\Pi(x,a\epsilon_k)$ we can take $(u,\gamma_0,k)$-admissible manifold
$H^u(y)\subset \Pi(x,a\epsilon_k)$ and $(s,\gamma_0,k)$-admissible
manifold $H^s(y)\subset \Pi(x,a\epsilon_k)$. Fixing $\gamma_0$ small
enough, we can assume that \\$(i)~ \|D_zf(v)\|''_{fz}\leq
e^{-\beta_1''+\epsilon}\|v\|''$ for every $v\in
T_zH^s(y), z\in H^s(y)$;\\
 $(ii)~ \|D_zf^{-1}(w)\|''_{f^{-1}z}\leq e^{-\beta_2''+\epsilon}\|w\|''¡¸$ for every $w\in
T_zH^u(y), z\in H^u(y)$.\\

For any regular point $x\in \Lambda$, define $k(x)=\min\{i\in
\mathbb{Z}\mid x\in \Lambda_i\}$. Using the local hyperbolicity
above, we can see that each connected component of $f(H^u(y))\cap
\Pi(fx,a\epsilon_{k(fx)})$ is an admissible
$(u,\gamma_0,k(fx))$-manifold; each connected component of
$f^{-1}(H^s(y))\cap \Pi(f^{-1}x,a\epsilon_{k(f^{-1}x)})$ is an
admissible $(s,\gamma_0,k(f^{-1}x))$-manifold.

\subsection{ Weak shadowing lemma}

 In this section, we state a weak shadowing
property for $C^{1+\alpha}$ non-uniformly hyperbolic systems, which
is needful in our proofs.

Let $(\delta_k)_{k=1}^{\infty}$ be a sequence of positive real
numbers.  Let $(x_n)_{n=-\infty}^{\infty}$ be a sequence of points
in $\Lambda=\Lambda(\beta_1,\beta_2,\epsilon)$ for which there
exists a sequence $(s_n)_{n=-\infty}^{+\infty}$ of positive integers
satisfying: \begin{eqnarray*}&(a)& x_n\in
\Lambda_{s_n},\,\,\forall n\in \mathbb{Z };\\[2mm]& (b)&  | s_n-s_{n-1} |\leq
1, \forall \,n\in \mathbb{Z};\\[2mm]
&(c)&  d(f(x_n), x_{n+1})\leq \delta_{s_n},\,\,\,\forall\, n\in
\mathbb{Z},\end{eqnarray*} then we call
$(x_n)_{n=-\infty}^{+\infty}$ a $(\delta_k)_{k=1}^{\infty}$
pseudo-orbit. Given $c>0$,
 a point $x\in M$ is an $\epsilon$-shadowing point for the
 $(\delta_k)_{k=1}^{\infty}$
pseudo-orbit if $d(f^n(x), x_{n})\leq
c\epsilon_{s_n}$,\,\,$\forall\,n\in \mathbb{Z}$,   where
$\epsilon_k=\epsilon_0e^{-\epsilon k}$ are given by the definition
of Lyapunov neighborhoods.
 \begin{Thm} (Weak shadowing lemma \cite{Hirayama,Katok1,Pollicott})\label{specification} Let $f: M\rightarrow M$ be a $C^{1+\alpha}$ diffeomorphism, with
a non-empty Pesin set $\Lambda=\Lambda(\beta_1,\beta_2;\epsilon)$
and fixed parameters, $\beta_1,\beta_2 \gg \epsilon > 0$. For $c >
0$ there exists a sequence $(\delta_k)_{k=1}^{\infty}$ such that for
any $(\delta_k)_{k=1}^{\infty}$ pseudo-orbit there exists a unique
$c$-shadowing point. \end{Thm}

\section{Entropy for non compact spaces}
In our settings the saturated sets are often non compact. In
\cite{Bowen4} Bowen gave the definition of topological entropy for
non compact spaces.  We state the definition in a slightly different
way and they are in fact equivalent.  Let $E\subset M$ and
$\mathcal{C}_n(E,\varepsilon)$ be the set of all finite or countable
covers of $E$ by the sets of form $B_m(x,\varepsilon)$ with $m\geq
n$. Denote
$$\mathcal{Y}(E; t,n,\varepsilon)=\inf\{ \sum_{B_m(x,\varepsilon)\in A}\,\,e^{-tm}\mid \,\,A\in \mathcal{C}_n(E,\varepsilon)\},$$
$$\mathcal{Y}(E; t,\varepsilon)=\lim_{n\rightarrow \infty}\gamma(E; t,n,\varepsilon).$$
Define $$h_{\top}(E;\varepsilon)=\inf \{t\mid\,\mathcal{Y}(E;
t,\varepsilon)=0\}=\sup \{t\mid\,\mathcal{Y}(E;
t,\varepsilon)=\infty\}$$
 and the topological entropy of $E$ is
 $$h_{\top}(E,f)=\lim_{\varepsilon\rightarrow 0} h_{\top}(E;\varepsilon).$$

The following formulas from \cite{PS}(Theorem 4.1(3)) are subcases
of Bowen's variational principle and true for general topological
setting.
\begin{Prop}\label{generical lemmas} Let $K\subset
\mathcal{M}_{inv}(M,f)$ be  non-empty, compact
 and connected. Then
 $$h_{\top}(f,G_K)\leq \inf\{h_\mu(f)\mid \mu\in K\}.$$
 In particular, taking $K=\{\mu\}$ one has
 $$h_{\top}(f,G_{\mu})\leq h_\mu(f).$$
\end{Prop}By the above proposition, to prove
Theorem \ref{main theorem of measure} and Theorem \ref{main theorem
of set}, it suffices to show the following theorems.

\begin{Thm}\label{main theorem of measure1} For every $\mu\in \mathcal{M}_{inv}(\widetilde{\Lambda},
f)$, we have
$$h_{\top}(f,G_{\mu})\geq h_\mu(f).$$
\end{Thm}

\begin{Thm}\label{main theorem of set1}Let $\eta=\{\eta_n\}$ be a sequence decreasing to zero and
$\mathcal{M}(\widetilde{\Lambda},\eta)\subset\mathcal{M}_{inv}(\widetilde{\Lambda},f)
$ be the set of measures with hyperbolic rate $\eta$. Given any
nonempty compactly connected set $K\subset
\mathcal{M}(\widetilde{\Lambda},\eta)$, we have
$$h_{\top}(f,G_K)\geq\inf\{h_{\mu}(f)\mid \mu\in K\}.$$

\end{Thm}
\begin{Rem}Let $\mu\in \mathcal{M}_{inv}(M,f)$ and $K\subset
\mathcal{M}_{inv}(M,f)$ be a nonempty compactly connect set. In
\cite{PS}, C. E. Pfister and W. G. Sullivan proved that\\ (1) with
almost product property(for detailed definition, see \cite {PS}), it
holds that
$$h_{\top}(f,G_{\mu})= h_\mu(f);$$
(2) with  almost product property plus uniform separation(for
detailed definition, see \cite {PS}), it holds that
$$h_{\top}(f,G_K)=\inf\{h_{\mu}(f)\mid \mu\in K\}.
$$
However, for nonuniformly hyperbolic systems, the shadowing and
separation are inherent from the weak hyperbolicity of Lyapunov
neighborhoods which varies in the index $k$ of Pesin blocks
$\Lambda_k$, hence in general  almost product property and uniform
separation both fail to be true.
\end{Rem}

\section{Proofs of Theorem \ref{main theorem of measure} and Theorem \ref{main theorem of measure1}}

In this section, we will verify Theorem \ref{main theorem of measure1} and thus complete the proof of Theorem \ref{main theorem of measure} by Proposition \ref{generical lemmas}.

 For
each ergodic measure $\nu$, we use Katok's definition of metric
entropy( see \cite{Katok2}). For $x,y\in M$ and  $n\in \mathbb{N}$,
let
$$d^n(x,y)=\max_{0\leq i\leq n-1}d(f^i(x),\,f^i(y)).$$ For
$\varepsilon,\,\delta>0$, let $N_n(\varepsilon,\,\delta)$ be the
minimal number of $\varepsilon-$ Bowen balls $B_n(x,\,\varepsilon)$
in the $d^n-$metric, which cover a set of $\nu$-measure at least
$1-\delta$. We define
$$h_{\nu}^{Kat}(f,\varepsilon\mid \delta)=\limsup_{n\rightarrow\infty}\frac{\log
N_n(\varepsilon,\,\delta)}{n}.$$ It follows by Theorem 1.1 of
\cite{Katok2} that
$$h_{\nu}(f)=\lim_{\varepsilon\rightarrow0}h_{\nu}^{Kat}(f,\varepsilon\mid \delta).$$
Recall that $\mathcal{M}_{erg}(M,f)$ denote the set of all ergodic
$f-$invariant measures supported on $M$. Assume
$\mu=\int_{\mathcal{M}_{erg}(M,f)}d\tau(\nu)$ is the ergodic
decomposition of $\mu$ then by Jacobs Theorem
$$h_{\mu}(f)=\int_{\mathcal{M}_{erg}(M,f)}h_{\nu}(f)d\tau(\nu).$$
Define
$$h_{\mu}^{Kat}(f,\varepsilon\mid \delta)\triangleq \int_{\mathcal{M}_{erg}(M,f)}h_{\nu}^{Kat}(f,\varepsilon\mid \delta)d\tau(\nu).$$
By Monotone Convergence Theorem, we have
$$h_{\mu}(f)=\int_{\mathcal{M}_{erg}(M,f)}\lim_{\varepsilon\rightarrow0}h_{\nu}^{Kat}(f,\varepsilon\mid \delta)d\tau(\nu)
=\lim_{\varepsilon\rightarrow0}h_{\mu}^{Kat}(f,\varepsilon\mid
\delta).$$

\noindent{\bf Proof of Theorem \ref{main theorem of
measure1}}\,\,\,Assume $\{\varphi_i\}_{i=1}^{\infty}$ is the dense
subset of $C(M)$ giving the weak$^*$ topology, that is,
$$D(\mu,\,\nu)=\sum_{i=1}^{\infty}\frac{|\int \varphi_id\mu-\int \varphi_id\nu|}{2^{i+1}\|\varphi_i\|}$$
for  $\mu,\nu\in\mathcal{M}(M)$. It is easy to check the affine
property of $D$, i.e., for any $\mu,\,m_1,\,m_2\in\mathcal{M}(M)$
and $0\leq\theta\leq1$,
\begin{eqnarray*}
D(\mu,\,\theta m_1+(1-\theta)m_2 )\leq \theta D(\mu,
m_1)+(1-\theta)D(\mu, m_2).\end{eqnarray*} In addition,
$D(\mu,\nu)\leq 1$ for any $\mu,\nu\in \mathcal{M}(M)$.  For any
integer $k\geq1$ and $\varphi_1,\cdots,\varphi_k$, there exists
$b_k>0$ such that
\begin{eqnarray}\label{points approximation}
d(\varphi_j(x),\varphi_j(y))<\frac{1}{k}\|\varphi_j\|\,\,\,\mbox{for
any}\,\,d(x,y)<b_k,\,\,1\leq j\leq k.\end{eqnarray}
Now fix
$\varepsilon,\delta>0$.

\begin{Lem}\label{rational approximation}For any integer $k\geq1$ and invariant measure $\mu$, we can take a
finite convex combination of ergodic probability measures with
rational coefficients,
$$\mu_k=\sum_{j=1}^{p_k}a_{k,j}\,m_{k,j}$$
such that
\begin{eqnarray}\label{measures approximation}
D(\mu,\mu_k)<\frac{1}{k},\,\,\,m_{k,j}(\widetilde{\Lambda})=1\,\,\,\mbox{
and}\,\,\, |h_{\mu}^{Kat}(f,\varepsilon\mid
\delta)-h_{\mu_k}^{Kat}(f,\varepsilon\mid
\delta)|<\frac1k.\end{eqnarray}
\end{Lem}
\begin{proof}
From the ergodic decomposition, we get
$$
\int_{\widetilde{\Lambda}} \varphi_i
d\mu=\int_{\mathcal{M}_{erg}(\widetilde{\Lambda}, f)}
\int_{\widetilde{\Lambda}} \varphi_i dm \,d\tau(m),\quad 1\leq i\leq
k.
$$
Use the definition of Lebesgue integral, we get the following steps.
First, we denote

 $$A_{+}:= \max_{1\leq i \leq
k}{\int_{\widetilde{\Lambda}} \varphi_i d m}+1, \,\,\,A_{-}:=
\min_{1\leq i \leq k}{\int_{\widetilde{\Lambda}} \varphi_i d m }-1$$
$$ F_{+}:=\sup_{\mathcal{M}_{erg}(\widetilde{\Lambda},f)}{h_{\mu}^{Kat}(f,\varepsilon\mid
\delta)}+1,\,\,\,
F_{-}:=\inf_{\mathcal{M}_{erg}(\widetilde{\Lambda},f)}{h_{\mu}^{Kat}(f,\varepsilon\mid
\delta)}-1.$$ It is easy to see that:
$$-\infty <A_{-}< A_{+}<+\infty,\,\,\,-\infty <F_{-}< F_{+}<+\infty.$$
For any integer $n>0$, let $$y_{0}=A_{-},\,\,
y_{j}-y_{j-1}=\frac{A_{+}-A_{-}}{n},\,\,y_{n}=A_{+},$$
$$F_{0}=F_{-},\,\,
F_{j}-F_{j-1}=\frac{F_{+}-F_{-}}{n},\,\,F_{n}=F_{+}.$$ We can take
$E_{i,j}$,$\,\,F_{s}\,\,$ to be  measurable partitions of
$\mathcal{M}_{erg}(\widetilde{\Lambda},f)\,\,$ as follows:
$$E_{i,j}=\{\mu\in \mathcal{M}_{erg}(\widetilde{\Lambda},f) \mid \quad y_{j}\leq \int_{\widetilde{\Lambda}} \varphi_i dm\leq y_{j+1}\}, $$
$$F_{n}=\{\mu\in \mathcal{M}_{erg}(\widetilde{\Lambda},f)\mid \quad  F_{j}\leq h_{\mu}^{Kat}(f,\varepsilon\mid
\delta) \leq F_{j+1}\}.$$  Noticing the fact that
$\bigcup_{j}E_{i,j}=\mathcal{M}_{erg}(\widetilde{\Lambda},f)$ and
$\bigcup_{j}F_{n}=\mathcal{M}_{erg}(\widetilde{\Lambda},f)$, we can
choose  a new partition $\,\xi\,$ defined as:
$$\xi=\bigwedge_{i,j}E_{i,j}\bigwedge_{s}F_{s},$$ where $\varsigma\bigwedge \zeta$
is given by $\{A\cap B\mid \,A\in \varsigma,\,\,B\in \zeta\}$.
For convenience, denote
$\xi=\{\xi_{k,1},\xi_{k,2},\cdots,\xi_{k,p_{k}}\}$.  To finish the
proof of Lemma \ref{rational approximation}, we can let $n$ large
enough such that any combination
$$\mu_k=\sum_{j=1}^{p_k}a_{k,j}\,m_{k,j}$$ where $m_{k,j}\in
\xi_{k,j}$, rational numbers $a_{k,j}>0$ with
$|a_{k,j}-\tau(\xi_{k,j})|<\frac{1}{2k}$,  satisfies:
\begin{eqnarray*}\label{measures approximation}
D(\mu,\mu_k)<\frac{1}{k},\,\,\,m_{k,j}(\widetilde{\Lambda})=1\,\,\,\mbox{
and}\,\,\, |h_{\mu}^{Kat}(f,\varepsilon\mid
\delta)-h_{\mu_k}^{Kat}(f,\varepsilon\mid \delta)|<\frac1k.
\end{eqnarray*}
\end{proof}
For each $k$, we can find $l_k$ such that
$m_{k,j}(\widetilde{\Lambda}_{l_k})>1-\delta$ for all $1\leq j\leq
p_k$. Recalling that $\epsilon_{l_k}$ is the scale of  Lyapunov
neighborhoods associated with the Pesin block $\Lambda_{l_k}$. For
any $x\in \Lambda_{l_k}$, $Df$ exhibits uniform hyperbolicity in
$B(x,\epsilon_{l_k})$. For $c=\frac{\varepsilon}{8\epsilon_0}$, by
Theorem \ref{specification} there is a sequence of numbers
$(\delta_k)_{k=1}^\infty$. Let $\xi_k$ be a finite partition of $M$
with $\mbox{diam}\,\xi_k
<\min\{\frac{b_k(1-e^{-\epsilon})}{4\sqrt{2}e^{(k+1)\epsilon}},\epsilon_{l_k},\delta_{l_k}\}$
and $\xi_k>\{\widetilde{\Lambda}_{l_k},M\setminus\widetilde{
\Lambda}_{l_k}\}$. Given $t\in \mathbb{N}$, consider the
set\begin{eqnarray*} \Lambda^t(m_{k,j})&=&\big{\{}x\in
\widetilde{\Lambda}_{l_k}\mid f^q(x)\in \xi_k(x)~ \mbox{for some} ~
q\in[t,[(1+{\frac1{k}})t]\\[2mm]
 &&\mbox{and}~
 D(\mathcal{E}_n(x),m_{k,j})<\frac1k\,\,\mbox{ for all}\,\,n\geq
 t\big{\}},
 \end{eqnarray*}
where $\xi_k(x)$ denotes the element in the partition $\xi_k$ which
contains the point $x$. Before going on the proof, we give the
following claim.
\smallskip

{\bf Claim}
$$m_{k,j}(\Lambda^t(m_{k,j}))\rightarrow m_{k,j}(\widetilde{\Lambda}_{l_k})\,\,\,\mbox{as}\,\,t\rightarrow +\infty.$$
\begin{proof}\,\,By ergodicity of $m_{k,j}$ and Birkhoff Ergodic Theorem, we know that for $m_{k,j}-a.e.x\in \widetilde{\Lambda}_{l_k}$, it holds $$\lim_{n\to\infty}\mathcal{E}_{n}(x)=m_{k,j}.$$ So we only need prove that the set $$ \Lambda^t_1(m_{k,j})=\big{\{}x\in
\widetilde{\Lambda}_{l_k}\mid f^q(x)\in \xi_k(x)~ \mbox{for some} ~
q\in[t,[(1+{\frac1{k}})t]\big{\}}$$ satisfying the property
 $$m_{k,j}(\Lambda^t_1(m_{k,j}))\rightarrow m_{k,j}(\widetilde{\Lambda}_{l_k})\,\,\,\mbox{as}\,\,t\rightarrow +\infty.$$

We next need the quantitative Poincar$\acute{e}$'s Recurrence Theorem (see Lemma 3.12 in \cite{Bochi} for more detail) as following.

\begin{Lem}\label{Reccurence Thm}
Let $f$ be a $C^1$ diffeomorphism preserving an invariant measure $\mu$ supported on $M$. Let $\Gamma\subset M$ be a measurable set with $\mu(\Gamma)> 0$ and let
$$\Omega =\cup_{n\in\mathbb{Z}}f^n(\Gamma).$$
Take $\gamma> 0$. Then there exists a measurable function $N_0 : \Omega\to \mathbb{N}$ such that for $a.e. x\in\Omega$,
every $n\geq N_0(x)$ and every $t\in [0, 1]$ there is some $l\in\{0, 1, . . .,n\}$ such that $f^l(x)\in\Gamma$
and $|(l/n)-t | < \gamma$.
\end{Lem}

\begin{Rem}
A slight modify (More precisely, replacing the interval $(n(t-\gamma),n(t+\gamma))$ by $(n(t,n(t+\gamma))$), one can require that $(l/n)-t<\gamma$ in the above lemma. Hence we have $l\in [n,n(t+\gamma)]$.
\end{Rem}

Take an element $\xi^l_k$ of the partition $\xi_k$. Let
$\Gamma=\xi^l_k$, $\gamma=\frac1k$. Applying Lemma \ref{Reccurence
Thm} and its remark, we can deduce that for $a.e.x\in\xi^l_k$, there
exists a measurable function $N_0$ such that for every $t\geq
N_0(x)$ there is some $q\in\{0, 1, \cdots ,n\}$ such that
$f^q(x)\in\xi^l_k=\xi_k(x)$ and $q\in [t,t(1+\frac1k)]$. That is to
say, $t\geq N_0(x)$ implies $x\in\Lambda^t_1(m_{k,j})$. And this
property holds for $a.e.x\in\xi^l_k$. Hence it is true for
$a.e.x\in\Lambda^l_k$. This completes the proof of the claim.

\end{proof}
\smallskip

Now we continue our proof of Theorem \ref{main theorem of measure1}. By above claim, we can take $t_k$ such that
$$m_{k,j}(\Lambda^{t}(m_{k,j}))>1-\delta$$ for all $t\geq t_k$ and
 $1\leq j\leq p_k$.

 Let $E_t(k,j)\subset \Lambda^{t}(m_{k,j})$ be a $(t,\varepsilon)$-separated set of
maximal cardinality. Then $\Lambda^{t}(m_{k,j})\subset \cup_{x\in
E_t(k,j)}B_{t}(x,\varepsilon) $, and by the definition of Katok's
entropy there exist infinitely many $t$ satisfying
$$\sharp\,E_t(k,j)\geq e^{t(h_{m_{k,j}}^{Kat}(f,\varepsilon\mid \delta)-\frac1k)}.$$
For each $q\in [t,[(1+{\frac1{k}})t]$, let
$$V_q=\{x\in E_t(k,j) \mid f^q(x)\in \xi_k(x)\}$$ and let $n=n(k,j)$ be
the value of $q$ which maximizes $\sharp\,V_q$. Obviously, \begin{eqnarray}\label{n t} t\geq
\frac{n}{1+\frac1k}\geq n(1-\frac{1}{k}).\end{eqnarray} Since
$e^{\frac{t}{k}}>\frac{t}{k}$, we deduce that
$$\sharp\, V_n\geq \frac{\sharp\,E_t(k,j)}{\frac{t}{k}}\geq e^{t(h_{m_{k,j}}^{Kat}(f,\varepsilon\mid \delta)-\frac3k)}.$$
Consider the element $A_n(m_{k,j})\in \xi_k$ for which
$\sharp\,(V_n\cap A_n(m_{k,j})) $ is maximal. It follows that
$$\sharp\,(V_n\cap A_n(m_{k,j}))\geq \frac{1}{\sharp\,\xi_k}\sharp\,V_n\geq
\frac{1}{\sharp\,\xi_k}e^{t(h_{m_{k,j}}^{Kat}(f,\varepsilon\mid
\delta)-\frac3k)}.$$ Thus taking $t$ large enough so that
$e^{\frac{t}{k}}>\sharp\,\xi_k$, we have by inequality (\ref{n t}) that
\begin{eqnarray}\label{count}\sharp\,(V_n\cap A_n(m_{k,j}))\geq e^{t(h_{m_{k,j}}^{Kat}(f,\varepsilon\mid
\delta)-\frac4k)}\geq
e^{n(1-\frac1k)(h_{m_{k,j}}^{Kat}(f,\varepsilon\mid
\delta)-\frac4k)}.\end{eqnarray}

Notice that  $A_{n(k,j)}(m_{k,j})$ is contained in an open subset
$U(k,j)$ of some Lyapunov neighborhood with
$\diam(U(k,j))<2\diam(\xi_k)$. By the ergodicity of $\omega$, for
any two measures $m_{k_1,j_1}, m_{k_2,j_2}$  we can find
$y=y(m_{k_1,j_1},m_{k_2,j_2})\in
U(k_1,j_1)\cap\widetilde{\Lambda}_{l_{k_1}}$ satisfying that for
some $s=s(m_{k_1,j_1},m_{k_2,j_2})$ one has
$$f^{s}(y)\in U(k_2,j_2)\cap\widetilde{\Lambda}_{l_{k_2}}.$$
Letting $C_{k,j}=\frac{a_{k,j}}{n(k,j)}$,  we can choose an integer
$N_k$ larger enough so that $N_kC_{k,j}$ are integers and
$$N_k\geq k\sum_{1\leq r_1,r_2\leq k+1, 1\leq j_i\leq p_{r_i},i=1,2}s(m_{r_1,j_1},m_{r_2,j_2}).$$
 Arbitrarily  take $x(k,j)\in A_n(m_{k,j})\cap
V_{n(k,j)}$. Denote  sequences
\begin{eqnarray*}
X_k&=&\sum_{j=1}^{p_k-1}s(m_{k,j},m_{k,j+1})+s(m_{k,p_k},m_{k,1})\\[2mm]
Y_k&=&\sum_{j=1}^{p_k}N_kn(k,j)C_{k,j}+X_k=N_k+X_k.\end{eqnarray*}
So,
\begin{eqnarray}\label{small bridge}\frac{N_k}{Y_k}\geq
\frac{1}{1+\frac1k}\geq 1-\frac1k.\end{eqnarray}
  We
further choose a strictly increasing sequence $\{T_k\}$ with $T_k\in
\mathbb{N}$,
\begin{eqnarray}\label{circle1}Y_{k+1}&\leq& \frac{1}{k+1} \sum
_{r=1}^{k}Y_rT_r,\\[2mm]
\label{circle2}\sum_{r=1}^{k}(Y_rT_r+s(m_{r,1},m_{r+1,1}))&\leq&
\frac{1}{k+1}Y_{k+1}T_{k+1}.\end{eqnarray}

In order to obtain shadowing points $z$ with our desired property
$\mathcal{E}_n(z)\rightarrow \mu$ as $n\rightarrow +\infty$,  we
first construct pseudo-orbits with satisfactory property in the
measure theoretic sense.
 For simplicity of the statement, for $x\in M$ define segments of orbits
\begin{eqnarray*}L_{k,j}(x)&\triangleq&(x,f(x),\cdots,
f^{n(k,j)-1}(x)),\,\,\,1\leq j\leq p_k,\\[2mm]
\widehat{L}_{k_1,j_1;k_2,j_2}(x)&\triangleq&(x,f(x)\cdots,
f^{s(m_{k_1,j_1},m_{k_2,j_2})-1}(x)),\,\,1\leq j_i\leq
p_{k_i},i=1,2.\end{eqnarray*}
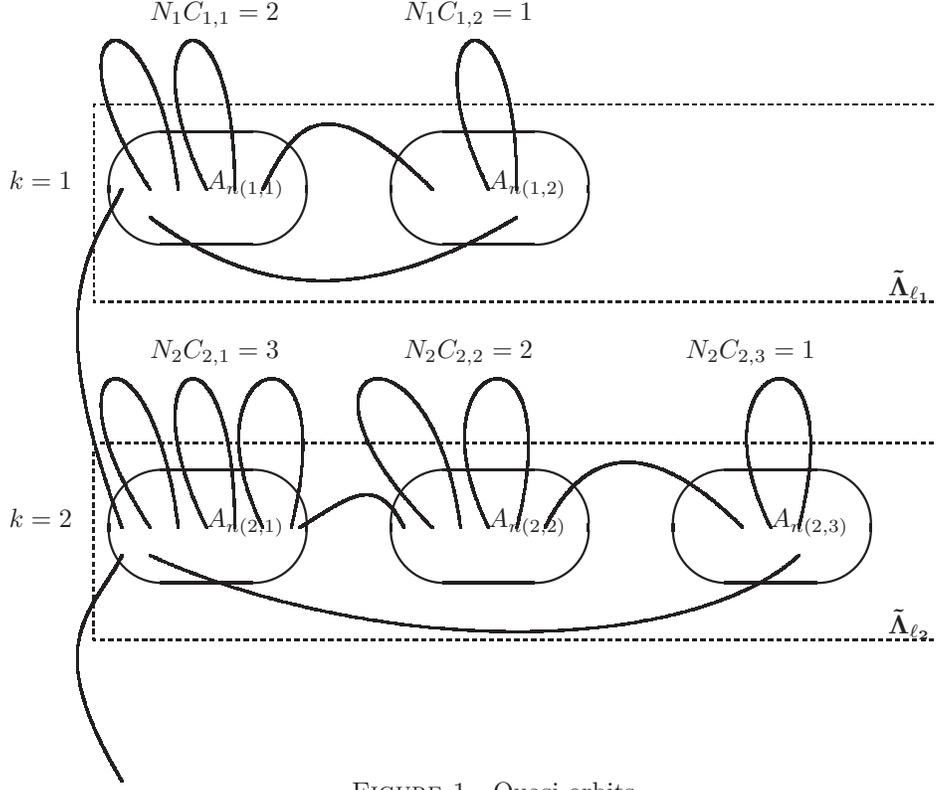
\begin{figure}[h]
\setlength{\unitlength}{0.75cm}
\begin{picture}(10,10)

\put(-2,4){\dashbox{0.075}(15,3.5)[br]{$\bf{\tilde\Lambda_{\ell_1}}$
}} \thicklines \put(0,6){\oval(3.5,2){$A_{n(1,1)}$}}
\put(5,6){\oval(3.5,2){$A_{n(1,2)}$}}
\linethickness{0.25mm}
\put(-2,-2){\dashbox{0.075}(15,3.5)[br]{$\bf{\tilde\Lambda_{\ell_2}}$
}} \thicklines \put(0,0){\oval(3.5,2){$A_{n(2,1)}$}}
\put(5,0){\oval(3.5,2){$A_{n(2,2)}$}}
\put(10,0){\oval(3.5,2){$A_{n(2,3)}$}}

\cbezier[500](-1,0)(-3,3)(-1,4)(-0.5,0)
\cbezier[500](-1,6)(-3,9)(-1,10)(-0.5,6)
\cbezier[500](0,0)(-1.5,3)(0.5,4)(0.5,0)
\cbezier[500](0,6)(-1.5,9)(0.5,10)(0.5,6)
\cbezier[500](5,6)(3.5,9)(5.5,10)(5.5,6)

\cbezier[500](4,0)(1,3)(4,4)(4.5,0)
\cbezier[500](1,0)(-0.5,3)(2.5,4)(1.5,0)
\cbezier[500](5,0)(3.5,3)(6.5,4)(5.5,0)
\cbezier[500](10,0)(8.5,3)(11.5,4)(10.5,0)
\cbezier[500](-1.5,0)(-3,4)(-2,5)(-1.5,6)
\cbezier[500](-1.5,-4.5)(-3,-2)(-2,-1.5)(-1.5,-0.5)

\cbezier[500](1,6)(1.5,7)(2,8)(4,6)
\cbezier[500](5.5,5.5)(3,4)(1,4)(-1,5.5)

\cbezier[500](1.65,0)(2.5,0.5)(3,1)(3.5,0)
\cbezier[500](6,0)(6.5,1)(7.5,2)(9.5,0)
\cbezier[500](10.5,-0.5)(8.5,-2.3)(3,-2.3)(-1,-0.5)

\put(-3.5,6){$k=1$} \put(-3.5,0){$k=2$} \put(-1,9){$N_1C_{1,1}=2$}
\put(3.5,9){$N_1C_{1,2}=1$} \put(-1,3){$N_2C_{2,1}=3$}
\put(3.5,3){$N_2C_{2,2}=2$} \put(8.5,3){$N_2C_{2,3}=1$}

\end{picture}
\bigskip
$$$$\\
\bigskip

\bigskip

\bigskip

\bigskip
\bigskip
\bigskip
\caption{\,Quasi-orbits }
\end{figure}

 Consider now the pseudo-orbit \begin{eqnarray}
 \label{quasi orbits}\quad  \quad O&=&O(x(1,1;1,1),\cdots,x(1,1; 1, N_1C_{1,1}),\cdots,x(1,p_1; 1, 1),\cdots,x(1,p_1; 1, N_1C_{1,p_1});\nonumber\\[2mm]
 &&\cdots;\nonumber \\[2mm]
 &&x(1,1;T_1,1),\cdots,x(1,1; T_1, N_1C_{1,1}),\cdots,x(1,p_1; T_11),\cdots,x(1,p_1; T_1, N_1C_{1,p_1});\nonumber\\[2mm]
 &&\vdots\nonumber \\[2mm]
 &&x(k,1; 1, 1),\cdots,x(k,1; 1, N_kC_{k,1}),\cdots,x(k,p_k,; 1, 1),\cdots,x(k,p_k; 1, N_kC_{k,p_k});\nonumber\\[2mm]&&\cdots;\nonumber \\[2mm]
 &&x(k,1;T_k,1),\cdots,x(K, 1; T_k, N_kC_{k,1}),\cdots,x(k,p_k; T_k, 1),\cdots,x(k,p_k; T_k, N_kC_{k,p_k});\nonumber\\[2mm]
 &&\vdots\nonumber \\[2mm]
 &&\cdots\,\,)\nonumber\end{eqnarray}
 with the precise form as follows
\begin{eqnarray*} \label{precise quasi orbits} &\big{\{}&\,\,[\,L_{1,1}(x(1,1; 1,1)),\cdots,L_{1,1}(x(1,1; 1, N_1C_{1,1})),\widehat{L}_{1,1,1,2}(y(m_{1,1},m_{1,2}));\\[2mm]
&&L_{1,2}(x(1,2; 1,1)),\cdots,L_{1,2}(x(1,2;1, N_1C_{1,2})),\widehat{L}_{1,2,1,3}(y(m_{1,1},m_{1,2})); \cdots\nonumber\\[2mm]
&&L_{1,p_1}(x(1,p_1;1, 1)),\cdots,L_{1,p_1}(x(1,p_1;1, N_1C_{1,p_1})),\widehat{L}_{1,p_1,1,1}(y(m_{1,p_1},m_{1,1}))\,\,;\nonumber \\[2mm]
&&\cdots\nonumber\\[2mm]
&&\,L_{1,1}(x(1,1; T_1, 1)),\cdots,L_{1,1}(x(1,1; T_1, N_1C_{1,1})),\widehat{L}_{1,1,1,2}(y(m_{1,1},m_{1,2}));\nonumber\\[2mm]
&&L_{1,2}(x(1,2;T_1, 1)),\cdots,L_{1,2}(x(1,2; T_1, N_1C_{1,2})),\widehat{L}_{1,2,1,3}(y(m_{1,1},m_{1,2})); \cdots\nonumber\\[2mm]
&&L_{1,p_1}(x(1,p_1;T_1, 1)),\cdots,L_{1,p_1}(x(1,p_1;T_1, N_1C_{1,p_1})),\widehat{L}_{1,p_1,1,1}(y(m_{1,p_1},m_{1,1}))\,]\,;\nonumber \\[2mm]
&&\widehat{L}(y(m_{1,1},m_{2,1}));\nonumber\\[2mm]
&&\vdots\nonumber\\[2mm]
&&\,[\,L_{k,1}(x(k,1; 1,1)),\cdots,L_{k,1}(x(k,1;1, N_kC_{k,1})),\widehat{L}_{k,1,k,2}(y(m_{k,1},m_{k,2}));\nonumber\\[2mm]
&&L_{k,2}(x(k,2; 1, 1)),\cdots,L_{k,2}(x(k,2;1, N_kC_{k,2})),\widehat{L}_{k,2,k,3}(y(m_{k,1},m_{k,2}));   \cdots\nonumber \\[2mm]
&&L_{k,p_k}(x(k,p_k; 1, 1)),\cdots,L_{k,p_k}(x(k,p_k;1, N_kC_{k,p_k})),\widehat{L}_{k,p_k,k,1}(y(m_{k,p_k},m_{k,1}))\,\,;\nonumber\\[2mm]
&&\widehat{L}(y(m_{k,1},m_{k+1,1}));\nonumber\\[2mm]
&&\cdots\,\,\,\nonumber\\[2mm]
&&L_{k,1}(x(k,1; T_k,1)),\cdots,L_{k,1}(x(k,1;T_k, N_kC_{k,1})),\widehat{L}_{k,1,k,2}(y(m_{k,1},m_{k,2}));\nonumber\\[2mm]
&&L_{k,2}(x(k,2; T_k, 1)),\cdots,L_{k,2}(x(k,2;T_k, N_kC_{k,2})),\widehat{L}_{k,2,k,3}(y(m_{k,1},m_{k,2}));   \cdots\nonumber \\[2mm]
&&L_{k,p_k}(x(k,p_k; T_k, 1)),\cdots,L_{k,p_k}(x(k,p_k;T_k, N_kC_{k,p_k})),\widehat{L}_{k,p_k,k,1}(y(m_{k,p_k},m_{k,1}))\,]\,;\nonumber\\[2mm]
&&\widehat{L}(y(m_{k,1},m_{k+1,1}));\nonumber\\[2mm]
&&\vdots\,\,\,\nonumber\\[2mm]
&&\cdots
 \big{\}},\end{eqnarray*}
where $x(k,j;i,t)\in V_{n(k,j)}\cap A_{n(k,j)}(m_{k,j})$.

For $k\geq 1$,  $1\leq i\leq T_k$, $1\leq j\leq p_k$, $t\geq 1$, let
$M_1=0$,
\begin{eqnarray*}M_k&=&M_{k,1}=\sum_{r=1}^{k-1}(T_rY_r+s(m_{r,1},m_{r+1,1})),\\[2mm]
M_{k,i}&=&M_{k,i,1}=M_k+(i-1)Y_k,
\\[2mm]
M_{k,i,j}&=& M_{k,i,j,1}=M_{k,i}
+\sum_{q=1}^{j-1}(N_k\,n(k,q)C_{k,q}+s(m_{k,q},m_{k,q+1})),\\[2mm]
M_{k,i,j,t}&=&M_{k,i,j}+(t-1)n(k,j).\end{eqnarray*}
 By
Theorem \ref{specification}, there exists a shadowing point $z$ of
$O$ such that
$$d(f^{M_{k,i,j,t}+q}(z),f^q(x(k,j;i, t)))<c\epsilon_0e^{-\epsilon l_k}<\frac{\varepsilon}{4\epsilon_0}\epsilon_0e^{-\epsilon l_k}
\leq \frac{\varepsilon}{4},$$ for $0\leq q\leq n(k,j)-1,$ $1\leq i
\leq T_k$, $1\leq t\leq N_kC_{k,j}$,  $1\leq j\leq p_k$. To be
precise,  $z$ can be considered as a map with variables
$x(k,j;i,t)$:
\begin{eqnarray}
 \label{quasi orbits}\quad  \quad z&=&z(x(1,1;1,1),\cdots,x(1,1; 1, N_1C_{1,1}),\cdots,x(1,p_1; 1, 1),\cdots,x(1,p_1; 1, N_1C_{1,p_1});\nonumber\\[2mm]&&\cdots;\nonumber \\[2mm]
 &&x(1,1;T_1,1),\cdots,x(1,1; T_1, N_1C_{1,1}),\cdots,x(1,p_1; T_11),\cdots,x(1,p_1; T_1, N_1C_{1,p_1});\nonumber\\[2mm]&&\cdots;\nonumber \\[2mm]
 &&x(k,1; 1, 1),\cdots,x(k,1; 1, N_kC_{k,1}),\cdots,x(k,p_k,; 1, 1),\cdots,x(k,p_k; 1, N_kC_{k,p_k});\nonumber\\[2mm]&&\cdots;\nonumber \\[2mm]
 &&x(k,1;T_k,1),\cdots,x(K, 1; T_k, N_kC_{k,1}),\cdots,x(k,p_k; T_k, 1),\cdots,x(k,p_k; T_k, N_kC_{k,p_k});\nonumber\\[2mm]&&\cdots;\nonumber \\[2mm]
 &&\cdots\,\,)\nonumber\end{eqnarray} We denote by $\mathcal{J}$ the set of
all shadowing points $z$ obtained in above procedure.

\begin{Lem}\label{converge} $\overline{\mathcal{J}}\subset G_{\mu}$. \end{Lem}
\begin{proof} First we prove  that for any $z\in \mathcal{J}$,
$$\lim_{k\rightarrow+\infty}\mathcal{E}_{M_{k}}(z)=\mu.$$
We begin by estimating $d(f^{M_{k,i,j,t}+q}(z),f^q(x(k,j;i,t)))$ for
$0\leq q\leq n(k,j)-1.$  Recalling that in the procedure of finding
the shadowing point $z$, all the constructions are done in the
Lyapunov neighborhoods $\Pi(x(k,j,t),a\epsilon_k)$. Moreover, notice
that we have required
$\mbox{diam}\,\xi_k<\frac{b_k(1-e^{-\epsilon})}{4\sqrt{2}e^{(k+1)\epsilon}}$
which implies that for every two adjacent orbit segments $x(k,j;
i_1, t_1)$ and $x(k,j; i_2, t_2)$, the ending point of the front
orbit segment and the beginning point of the segment following are
$\frac{b_k(1-e^{-\epsilon})}{4\sqrt{2}e^{(k+1)\epsilon}}$ close to
each other. Let $y$ be the unique intersection point of admissible
manifolds $H^s(z)$ and $H^u(x)$. In what follows, define $d''$ to be
the distance induced by $\|\cdot\|''$ in the local Lyapunov
neighborhoods. By the hyperbolicity of $Df$ in the Lyapunov
coordinates \footnote{This hyperbolic property is crucial in the
estimation of distance along adjacent segments, so the weak
shadowing lemma \ref{specification} (which is actually stated in
topological way) does not suffice to conclude Theorem \ref{main
theorem of measure} and the following Theorem \ref{main theorem of
set}.}, we obtain
\begin{eqnarray*}&&d(f^{M_{k,i,j,t}+q}(z),f^q(x(k,j; i, t)))\\[2mm]
 &\leq&
d(f^{M_{k,i,j,t}+q}(z),f^q(y))+d(f^q(y),f^q(x(k,j; i, t)))\\[2mm]
&\leq&\sqrt{2}d''(f^{M_{k,i,j,t}+q}(z),f^q(y))+\sqrt{2}d''(f^q(y),f^q(x(k,j; i, t)))\\[2mm]
&\leq&\sqrt{2}e^{-(\beta_1''-\epsilon)q}d''(f^{M_{k,i,j,t}}(z),y)+\sqrt{2}e^{-(\beta_2''-\epsilon)(n(k,j)-q)}d''(f^{n(k,j)}(y),f^{n(k,j)}(x(k,j; i, t)))\\[2mm]
&\leq &\sqrt{2}
\max\{e^{-(\beta_1''-\epsilon)q},e^{-(\beta_2''-\epsilon)(n(k,j)-q)}\}(d''(f^{M_{k,i,j,t}}(z),y)\\[2mm]&&+d''(f^{n(k,j)}(y),f^{n(k,j)}(x(k,j; i, t))))\\[2mm]
&\leq&\frac{2\sqrt{2}e^{\epsilon
(k+1)}}{1-e^{-\epsilon}}(d(f^{M_{k,i,j,t}}(z),y)+d(f^{n(k,j)}(y),f^{n(k,j)}(x(k,j; i, t))))\\[2mm]
&\leq& \frac{2\sqrt{2}e^{\epsilon
(k+1)}}{1-e^{-\epsilon}} 2\diam (\xi_k)\\[2mm]
&<&b_k
\end{eqnarray*} for
$0\leq q\leq n(k,j)-1$. Now we can deduce that
$$|\varphi_p(f^{M_{k,i,j,t}+q}(z)-\varphi_p(f^q(x(k,j; i,t)))|<\frac{1}{k}\|\varphi_p\|,\,\,\,\,1\leq p\leq k,$$
which implies that
\begin{eqnarray}\label{approximation2}D(\mathcal{E}_{n(k,j)}(f^{M_{k,i,j,t}}(z)),\mathcal{E}_{n(k,j)}(x(k,j; i, t)))<\frac1k+\frac{1}{2^{k-1}}<\frac2k,
\end{eqnarray}
for sufficiently large $k$. By the triangle inequality, we have
\begin{eqnarray*}
D(\mathcal{E}_{Y_k}(f^{M_{k,i}}(z)),\,\mu)&\leq&
D(\mathcal{E}_{Y_k}(f^{M_{k,i}}(z)),\,\mu_k)+\frac1k\\[2mm]
&\leq&D(\mathcal{E}_{Y_k}(f^{M_{k,i}}(z)),\,\frac{1}{Y_k-X_k}\sum_{j=1}^{p_k}N_kC_{k,j}n(k,j)\mathcal{E}_{n(k,j)}(f^{M_{k,i,j}}(z)))\\[2mm]
&&+D(\frac{1}{Y_k-X_k}\sum_{j=1}^{p_k}N_kC_{k,j}n(k,j)\mathcal{E}_{n(k,j)}(f^{M_{k,i,j}}(z)),\,\mu_k)+\frac1k
\end{eqnarray*}
Note that for any $\varphi\in C^0(M)$, it holds
\begin{eqnarray*}
& &\|\int\varphi d\mathcal{E}_{Y_k}(f^{M_{k,i}}(z))-\int\varphi d\frac{1}{Y_k-X_k}\sum_{j=1}^{p_k}N_kC_{k,j}n(k,j)\mathcal{E}_{n(k,j)}(f^{M_{k,i,j}}(z)))\|\\[2mm]
&=&\|\frac1{Y_k}\sum_{q=1}^{Y_k-1}\varphi(f^{M_{k,i}+q}(z))-\frac1{Y_k-X_k}\sum_{j=1}^{p_k}N_kC_{k,j}\sum_{q=1}^{n(k,j)-1}\varphi(f^{M_{k,i,j}+q}(z))\|\\[2mm]
&\leq&\|\frac1{Y_k}\sum_{j=1}^{p_k}N_kC_{k,j}\sum_{q=1}^{n(k,j)-1}\varphi(f^{M_{k,i,j}+q}(z))-\frac1{Y_k-X_k}\sum_{j=1}^{p_k}N_kC_{k,j}\sum_{q=1}^{n(k,j)-1}\varphi(f^{M_{k,i,j}+q}(z))\|\\[2mm]
&&+\|\frac1{Y_k}(\sum_{j=1}^{p_k-1}\sum_{q=1}^{s(m_{k,j},\,m_{k,j+1})-1}\varphi(f^{M_{k,i,j}-s(m_{k,j},\,m_{k,j+1})+q}(z))\\[2mm]
&&+\sum_{q=1}^{s(m_{k,p_k},\,m_{k,1})-1}\varphi(f^{M_{k,i,j}-s(m_{k,p_k},\,m_{k,1})+q}(z)))\|\\[2mm]
&\leq&[|(\frac{1}{Y_k}-\frac{1}{Y_k-X_k})(Y_k-X_k)|+\frac{X_k}{Y_k}]\|\varphi\|.
\end{eqnarray*}
Then by the definition of $D$, the above inequality implies that
\begin{eqnarray*}
&&D(\mathcal{E}_{Y_k}(f^{M_{k,i}}(z)),\,\frac{1}{Y_k-X_k}\sum_{j=1}^{p_k}N_kC_{k,j}n(k,j)\mathcal{E}_{n(k,j)}(f^{M_{k,i,j}}(z)))\\&\leq&
|(\frac{1}{Y_k}-\frac{1}{Y_k-X_k})(Y_k-X_k)|+\frac{X_k}{Y_k}.
\end{eqnarray*}
Thus, by the affine property of $D$, together
with the property $a_{k,j}=n(k,j)C_{k,j}$ and $N_k=Y_k-X_k$, we have
\begin{eqnarray*}
D(\mathcal{E}_{Y_k}(f^{M_{k,i}}(z)),\,\mu)&\leq&
D(\frac{1}{Y_k-X_k}\sum_{j=1}^{p_k}N_kC_{k,j}n(k,j)\mathcal{E}_{n(k,j)}(f^{M_{k,i,j}}(z)),
\sum_{j=1}^{p_k}a_{k,j}m_{k,j})\\[2mm]&&+|(\frac{1}{Y_k}-\frac{1}{Y_k-X_k})(Y_k-X_k)|+\frac{X_k}{Y_k}+\frac{1}{k}\\[2mm]
&\leq&\frac{N_k}{Y_k-X_k}\sum_{j=1}^{p_k}a_{k,j}D(\mathcal{E}_{n(k,j)}(f^{M_{k,i,j}}(z),m_{k,j})+\frac{2X_k}{Y_k}+\frac{1}{k}\\[2mm]
&=&\sum_{j=1}^{p_k}a_{k,j}D(\mathcal{E}_{n(k,j)}(f^{M_{k,i,j}}(z),m_{k,j})+\frac{2X_k}{Y_k}+\frac{1}{k}.
\end{eqnarray*}
Noting that
\begin{eqnarray*}
&&\sum_{j=1}^{p_k}a_{k,j}D(\mathcal{E}_{n(k,j)}(f^{M_{k,i,j}}(z),m_{k,j})\\[2mm]
&\leq&\sum_{j=1}^{p_k}a_{k,j}D(\mathcal{E}_{n(k,j)}(f^{M_{k,i,j}}(z)),\mathcal{E}_{n(k,j)}(x(k,j)))
+\sum_{j=1}^{p_k}a_{k,j}D(\mathcal{E}_{n(k,j)}(x(k,j)),\,m_{k,j})
\end{eqnarray*}
and by the definition of $\Lambda^t(m_{k,j})$ which all $x(k,j)$
belong to and by (\ref{approximation2}), we can further deduce that
\begin{eqnarray*}
D(\mathcal{E}_{Y_k}(f^{M_{k,i}}(z)),\,\mu)&\leq&\sum_{j=1}^{p_k}a_{k,j}D(\mathcal{E}_{n(k,j)}(f^{M_{k,i,j}}(z),\mathcal{E}_{n(k,j)}(x(k,j)))
+\frac1k+\frac{2X_k}{Y_k}+\frac{1}{k}\\[2mm]
&\leq&\frac{2}{k}+\frac1k+\frac{2X_k}{Y_k}+\frac{1}{k}\\[2mm]
&\leq&\frac6k\,\,\,\,\,\,(\mbox{by}\,(\ref{small bridge})).
\end{eqnarray*}
Hence, by affine property and inequalities (\ref{circle1}) and (\ref{circle2}) and $D(\cdot,\cdot)\leq1$, we obtain that
\begin{eqnarray*}
D(\mathcal{E}_{M_{k+1}}(z),\,\mu)&\leq&
\frac{\sum_{r=1}^{k-1}(T_rY_r+s(m_{r,1},m_{r+1,1}))+s(m_{k,1},m_{k+1,1})}{T_kY_k+\sum_{r=1}^{k-1}(T_rY_r+s(m_{r,1},m_{r+1,1}))+s(m_{k,1},m_{k+1,1})}
\\[2mm]
&&+\frac{T_kY_k}{T_kY_k+\sum_{r=1}^{k-1}T_rY_r+s(m_{r,1},m_{r+1,1})}D(\mathcal{E}_{Y_k}(f^{M_{k,i}}(z)),\,\mu)\\[2mm]
&\leq&
\frac{\sum_{r=1}^{k-1}(T_rY_r+s(m_{r,1},m_{r+1,1}))+s(m_{k,1},m_{k+1,1})}{T_kY_k+\sum_{r=1}^{k-1}(T_rY_r+s(m_{r,1},m_{r+1,1}))+s(m_{k,1},m_{k+1,1})}
\\[2mm]
&&+\frac{T_kY_k}{T_kY_k+\sum_{r=1}^{k-1}T_rY_r+s(m_{r,1},m_{r+1,1})}\frac6k\\[2mm]
&\leq&\frac8k.
\end{eqnarray*}
Thus,
$$\lim_{k\rightarrow+\infty}\mathcal{E}_{M_k}(z)=\mu.$$
For $M_{k,i}\leq n\leq M_{k,i+1}$ (here we appoint
$M_{k,p_k+1}=M_{k+1,1}$), it follows that
\begin{eqnarray*}
D(\mathcal{E}_{n}(z),\,\mu)&\leq&
\frac{M_{k}}{n}D(\mathcal{E}_{M_{k}}(z),\mu)+\frac{1}{n}
\sum_{p=1}^{i-1}D(\mathcal{E}_{Y_k}(f^{M_{k,p-1}}(z)),\mu)\,\,\,\,(\mbox{by affine property})\\[2mm]
&&+\frac{n-M_{k,i}}{n}D(\mathcal{E}_{n-M_{k,i}}(f^{M_{k,i}}(z)),\mu)
\\[2mm]
&\leq&
\frac{M_{k}}{n}\frac8k+\frac{(i-1)Y_k}{n}\frac6k+\frac{Y_k+s(m_{k,1},m_{k+1,1})}{n}\\[2mm]
&\leq&\frac{15}k\,\,\,\,\,\,(\mbox{by}\,(\ref{circle1})
\,\mbox{and}\,(\ref{circle2})).
\end{eqnarray*}
Let $n\rightarrow +\infty$, then $k\rightarrow +\infty$ and
$\mathcal{E}_{n}(z)\rightarrow \mu$. That is $\mathcal{J}\in
G_{\mu}$.  For any $z'\in \overline{\mathcal{J}}$, we take $z_t\in
\mathcal{J} $ with $\lim_n z_t=z'$.  Observing that for $M_{k,i}\leq
n\leq M_{k,i+1}$, $D(\mathcal{E}_{n}(z_t),\,\mu)\leq 15/k$ by
continuity it also holds that $D(\mathcal{E}_{n}(z'),\,\mu)\leq
15/k$. This completes the proof of the Lemma \ref{converge}.
\end{proof}

To finish the proof of Theorem \ref{main theorem of measure1}, we
need to compute the entropy of $\overline{\mathcal{J}}\subset
G_{\mu}$.  Notice that the choices of the position labeled by
$x(k,j; i, t)$ in (\ref{quasi orbits}) has at least
$$e^{n(k,j)(1-\frac1k)(h_{m_{k,j}}^{Kat}(f,\varepsilon\mid
\delta)-\frac{4}{k})}$$ by (\ref{count}). Moreover, fixing the
position indexed $k,j,t$, for distinct $x(k,j; i, t), x'(k,j; i,
t)\in V_{n(k,j)}\cap A_{n(k,j)}(m_{k,j}) $, the corresponding
shadowing points $z,z'$ satisfying
\begin{eqnarray*}&&d(f^{M_{k,i,j,t}+q}(z),f^{M_{k,i,j,t}+q}(z'))\\&\geq&
d(f^{q}(x(k,j; i, t)), f^{q}(x'(k,j; i, t)))-d(f^{M_{k,i,j,t}+q}(z),
f^{q}(x(k,j; i, t)))\\&&-d(f^{M_{k,i,j,t}+q}(z'),
f^{q}(x'(k,j;i, t)))\\
&\geq& d(f^{q}(x(k,j; i, t)),
f^{q}(x'(k,j,t)))-\frac\varepsilon2.\end{eqnarray*} Since $x(k,j,t),
x'(k,j; i, t)$ are $(n(k,j),\varepsilon)$-separated, so
$f^{M_{k,i,j,t}}(z)$, $f^{M_{k,i,j,t}}(z')$ are
$(n(k,j),\frac\varepsilon2)$-separated. Denote sets concerning the
choice of quasi-orbits in  $M_{ki}$
\begin{eqnarray*}
H_{ki}=\{&&(x(k,j;i,1),\cdots,x(k,j;i,N_kC_{k,j}),\cdots,x(k,p_k;i,1),\cdots,x(1,p_k;i, N_kC_{k,p_k})\\[2mm]
 &&\mid \,\,x(k,j;i,t)\in V_{n(k,j)}\cap A_{n(k,j)}\}.
\end{eqnarray*}
 Then
\begin{eqnarray*}
\sharp H_{ki}\geq
e^{\sum_{j=1}^{p_k}N_kC_{k,j}n(k,j)(1-\frac1k)(h_{m_{k,j}}^{Kat}(f,\varepsilon\mid
\delta)-\frac{4}{k})}.
\end{eqnarray*}
Hence,
\begin{eqnarray}
\label{large katok entropy}\frac{1}{Y_k}\log \,\sharp H_{ki}&\geq&
\frac{Y_k-X_k}{Y_k}\sum_{j=1}^{p_k}a_{k,j}(1-\frac1k)(h_{m_{k,j}}^{Kat}(f,\varepsilon\mid\delta)-\frac{4}{k})\\[2mm]
&\geq&(1-\frac{1}{k})\sum_{j=1}^{p_k}a_{k,j}(1-\frac1k)(h_{m_{k,j}}^{Kat}(f,\varepsilon\mid\delta)-\frac{4}{k})\nonumber\\[2mm]
&=&(1-\frac1k)^2h_{\mu_{k}}^{Kat}(f,\varepsilon\mid\delta)-\frac4k(1-\frac{1}{k})^2\nonumber\\[2mm]
&\geq&(1-\frac1k)^2(h_{\mu}^{Kat}(f,\varepsilon\mid\delta)-\frac1k)-\frac4k(1-\frac{1}{k})^2.\nonumber
\end{eqnarray}
Since $\overline{\mathcal{J}}$ is compact we can take only finite
covers $\mathcal{C}(\overline{\mathcal{J}},\varepsilon/2)$ of
$\overline{\mathcal{J}}$ in the calculation of topological entropy
$h_{\top}(\overline{\mathcal{J}},\frac\varepsilon2)$. Let
$r<h_{\mu}^{Kat}(f,\varepsilon\mid\delta)$. For each $\mathcal{A}\in
\mathcal{C}(\overline{\mathcal{J}},\varepsilon/2)$ we define a new
cover $\mathcal{A}'$ in which for $M_{k,i}\leq m\leq M_{k,i+1}$,
$B_m(z,\varepsilon/2)$ is  replaced by
$B_{M_{k,i}}(z,\varepsilon/2)$, where we suppose $M_{k,0}=M_{k-1,
p_{k-1}}$,  $M_{k,p_k+1}=M_{k+1,1}$.  Therefore,
\begin{eqnarray*}
\mathcal{Y}(\overline{\mathcal{J}};r,n,\varepsilon/2
)=\inf_{\mathcal{A}\in
\mathcal{C}(\overline{\mathcal{J}},\varepsilon/2)}\sum_{B_{m}(z,\varepsilon/2)\in
\mathcal{A}} e^{-rm}\geq \inf_{\mathcal{A}\in
\mathcal{C}(\overline{\mathcal{J}},\varepsilon/2)}\sum_{B_{M_{k,i}}(z,\varepsilon/2)\in
\mathcal{A}'} e^{-rM_{k,i+1}}.
\end{eqnarray*}
Denote
$$b=b(\mathcal{A}')=\max\{M_{k,i}\mid\,\,B_{m}(z,\frac\varepsilon2)\in \mathcal{A}'\quad \mbox{and}\quad M_{k,i}\leq m<M_{k,i+1}\}.$$
Noticing that $\mathcal{A}'$ is a cover of $\mathcal{J}$ each point
of $\mathcal{J}$ belongs to some $B_{M_{k,i}}(x, \frac\varepsilon2)$
with  $M_{k,i}\leq b$. Moreover, if $z,z'\in \mathcal{J}$ with some
position $x(k,j;i,t)\neq x'(k,j;i,t)$ then $z,z'$ can't stay in the
same $B_{M_{k,i}}(x, \frac\varepsilon2)$.  Define
\begin{eqnarray*}
W_{k,i}=\{B_{M_{k,i}}(z, \frac\varepsilon2)\in \mathcal{A}' \}.
\end{eqnarray*}
It follows that
\begin{eqnarray*}
\sum_{M_{k,i}\leq b} \,\,\,\sharp W_{k,i} \,\,\Pi_{M_{k,i}<
M_{k',i'}\leq b } \,\,\,\sharp H_{k',i'} \,\,\geq\,\,
\Pi_{M_{k',i'}\leq b }\,\, \sharp H_{k',i'}\,\,.
\end{eqnarray*}
So,
\begin{eqnarray*}
\sum_{M_{k,i}\leq b} \,\,\,\sharp W_{k,i}\,\, (\Pi_{ M_{k',i'}\leq
M_{k,i} } \,\,\,\sharp H_{k',i'})^{-1} \,\,\geq\,\, 1.
\end{eqnarray*}
From (\ref{large katok entropy}) it is easily seen that
$$\limsup_{k\rightarrow \infty}\frac{\Pi_{ M_{k',i'}\leq M_{k,i} }
\,\,\,\sharp
H_{k',i'}}{\exp(h_{\mu}^{Kat}(f,\varepsilon\mid\delta)M_{k,i})}\geq
1.$$ Since $r<h_{\mu}^{Kat}(f,\varepsilon\mid\delta)$ and
$\lim_{k\rightarrow \infty}\frac{M_{k,i+1}}{M_{k,i}}=1$, we can take
$k$ large enough so that $$\frac{M_{k,i+1}}{M_{k,i}}\leq
\frac{h_{\mu}^{Kat}(f,\varepsilon\mid\delta)}{r}.$$ Thus  there is
some constant $c_0>0$ for large $k$
\begin{eqnarray*}
\sum_{B_{M_{k,i}}(z,\varepsilon/2)\in \mathcal{A}'}
e^{-rM_{k,i+1}}&=& \sum_{M_{k,i}\leq b} \,\,\,\sharp
W_{k,i}\,\,\,e^{-rM_{k,i+1}}\\[2mm]&\geq& \sum_{M_{k,i}\leq b} \,\,\,\sharp
W_{k,i}\,\, \exp(-h_{\mu}^{Kat}(f,\varepsilon\mid\delta)M_{k,i})
\,\,\\[2mm]&\geq& c_0\sum_{M_{k,i}\leq b} \,\,\,\sharp
W_{k,i}\,\,(\Pi_{ M_{k',i'}\leq M_{k,i} } \,\,\,\sharp
H_{k',i'})^{-1}\,\, \\[2mm]&\geq&c_0,
\end{eqnarray*}
which together with the arbitrariness of $r$ gives rise to the
required inequality
$$h_{\top}(\overline{\mathcal{J}},\frac\varepsilon2)\geq
h_{\mu}^{Kat}(f,\varepsilon\mid\delta).$$ Finally, the arbitrariness
of $\varepsilon$ yields:
$$h_{\top}(f,G_{\mu})\geq h_{\mu}(f).$$\hfill$\Box$

\section{Proofs of Theorem \ref{main theorem of set} and Theorem \ref{main theorem of set1}}
We start this section by recalling the notion  of entropy introduced
by Newhouse \cite{Newhouse}. Given $\mu\in \mathcal{M}_{inv}(M,f)$,
let $F\subset M$ be a measurable
set. Define \\
\begin{eqnarray*}&(1)&H(n,\rho\mid x,F,\varepsilon)=\log \max\{\sharp E\mid E
\,\mbox{is a}\,(d^n,\rho)-\mbox{separated set in }\,F\cap
B_{n}(x,\varepsilon) \};\\[2mm]
&(2)&H(n,\rho\mid F,\varepsilon)=\sup_{x\in F}H(n,\rho\mid x,
F,\varepsilon);
\\[2mm]
&(3)&h(\rho\mid
F,\varepsilon)=\limsup_{n\rightarrow+\infty}\frac1nH(n,\rho\mid
F,\varepsilon);\\[2mm]
&(4)&h( F,\varepsilon)=\lim_{\rho\rightarrow0}h(\rho\mid
F,\varepsilon);\\[2mm]
&(5)&h^{New}_{\loc}(
\mu,\varepsilon)=\liminf_{\sigma\rightarrow1}\{h(
F,\varepsilon)\mid \mu(F)>\sigma\};\\[2mm]
&(6)&h^{New}(\mu,\varepsilon)=h_{\mu}(f)-h^{New}_{\loc}(
\mu,\varepsilon)\end{eqnarray*}

Let $\{\theta_k\}_{k=1}^{\infty}$ be a decreasing sequence which
approaches zero. One can verify that
$(h^{New}(\mu,\theta_k)\mid_{\mu\in
\mathcal{M}_{inv}(M,f)})_{k=1}^{\infty}$ is in fact an increasing
sequence of functions defined on $\mathcal{M}_{inv}(M,f)$. Further
more,
$$\lim_{\theta_k\rightarrow
0}h^{New}(\mu,\theta_k)=h_{\mu}(f)\,\,\,\mbox{ for any}\,\,\,\mu\in
\mathcal{M}(f).$$

Let $\mathcal{H}=(h_k)$ and $\mathcal{H}'=(h_k')$ be two increasing
sequences of functions on a compact domain $\mathcal{D}$. We say
$\mathcal{H}'$ {\it uniformly dominates} $\mathcal{H}$, denoted by
$\mathcal{H}'\geq \mathcal{H}$, if for every index $k$ and every
$\gamma>0$ there exists an index $k'$ such that $$h'_{k'}\geq
h_k-\gamma.$$ We say that $\mathcal{H}$ and $\mathcal{H}'$ are {\it
uniformly equivalent} if both $\mathcal{H}\geq \mathcal{H}'$ and
$\mathcal{H}'\geq \mathcal{H}$. Obviously, uniform equivalence is an
equivalence relation.

Next we give some elements from the theory of entropy structures as
developed by Boyle-Downarowicz \cite{BD}. An increasing sequence
$\alpha_1\leq \alpha_2\leq\cdots$ of partitions of $M$ is called
{\it essential} (for $f$ ) if \begin{eqnarray*}&&(1)
\diam(\alpha_k)\rightarrow0 \,\,\mbox{as}\,\, k\rightarrow
+\infty,\\[2mm]
&&(2)\mu(\partial\alpha_k)= 0 \,\,\mbox{for every }\,\,\mu\in
\mathcal{M}_{inv}(M, f).\end{eqnarray*} Here $\partial\alpha_k$
denotes the union of the boundaries of elements in the partition
$\alpha_k$. Note that essential sequences of partitions may not
exist (e.g., for the identity map on the unit interval). However,
for any finite entropy system $(f, M)$ it follows from the work of
Lindenstrauss and Weiss \cite{Lindenstrauss}\cite{Lin-Weiss} that
the product $f\times R$ with $R$ an irrational rotation has
essential sequences of partitions. Noting that the rotation doesn't
contribute entropy for every invariant measure, we can always assume
$(f,M)$ has an essential sequence.  By an {\it entropy structure} of
a finite topological entropy dynamical system $(f,M)$ we mean an
increasing sequence $\mathcal{H}=(h_k)$ of functions defined on
$\mathcal{M}_{inv}(M,f)$ which is uniformly equivalent to
$(h_{\mu}(f,\alpha_k)\mid_{\mu\in \mathcal{M}_{inv}(M,f)})$.
Combining with Katok's definition of entropy, we consider an
increasing sequence of functions on $\mathcal{M}_{inv}(M,f)$ given
by $(h_{\mu}^{Kat}(f,\epsilon_k\mid\delta)\mid_{\mu\in
\mathcal{M}_{inv}(f)})$.
\begin{Thm}\label{entropy structure}
Both $(h_{\mu}^{Kat}(f,\theta_k\mid\delta)\mid_{\mu\in
\mathcal{M}_{inv}(M,f)})$ and $(h^{New}(\mu,\theta_k)\mid_{\mu\in
\mathcal{M}_{inv}(M,f)})$ are entropy structures hence they are
uniformly equivalent.
\end{Thm}
\begin{proof}This theorem is a part of Theorem 7.0.1 in \cite{Downarowicz}.\end{proof}
\begin{Rem}The entropy structure in fact reflects the uniform convergence of entropy. It is well known that there are various notions of entropy.
 However, not all of them can form entropy structure, for
 example, the classic definition of entropy by partitions (see Theorem 8.0.1 in \cite{Downarowicz}).\end{Rem}

Let $\eta=\{\eta_n\}_{n=1}^{\infty}$ be a sequence decreasing to  zero.
$\mathcal{M}(\widetilde{\Lambda},\eta)$ is the subset of
$\mathcal{M}_{inv}(M,f)$ with respect to the hyperbolic rate $\eta$.

For $\delta, \varepsilon>0$ and any $\Upsilon\subset
\mathcal{M}(M)$, define
$$h_{\Upsilon,\loc}^{Kat}(f,\varepsilon\mid \delta)=\max_{\mu\in
\Upsilon}\{h_{\mu}(f)-h_{\mu}^{Kat}(f,\varepsilon\mid \delta)\}.$$

\begin{Lem}\label{asymp entropy expansive}
$\lim_{\theta_k\rightarrow
0}h_{\mathcal{M}(\widetilde{\Lambda},\eta),\loc}^{Kat}(f,\theta_k\mid
\delta)=0$.
\end{Lem}

\begin{proof} First we need a proposition contained in Page 226 of
\cite{Newhouse}, which  reads as
$$\lim_{\varepsilon\rightarrow0}\sup_{\mu\in \mathcal{M}(\widetilde{\Lambda},\eta)}h_{\loc}^{New}(\mu,\varepsilon)=0.$$
By Theorem \ref{entropy structure},
$$ (h^{New}(\mu,\theta_k)\mid_{\mu\in \mathcal{M}_{inv}(M, f)})\leq(h_{\mu}^{Kat}(f,\theta_k\mid\delta)\mid_{\mu\in \mathcal{M}_{inv}(M, f)})
.$$ So, for any $k\in \mathbb{N}$,  there exists $k'>k$ such that
$$h_{\mu}^{Kat}(f,\theta_{k'}\mid\delta)
\geq h^{New}(\mu,\theta_k)-\frac1k,$$ for all $\mu\in
\mathcal{M}(\widetilde{\Lambda},\eta)$. It follows that
\begin{eqnarray*}
h_{\mu}(f)-h_{\mu}^{Kat}(f,\theta_{k'}\mid \delta)&\leq&
h_{\mu}(f)-(h^{New}(\mu, \theta_k)-\frac1k)\\[2mm]&=&h_{\loc}^{New}(\mu,\theta_k)+\frac1k,
\end{eqnarray*}
for all $\mu\in \mathcal{M}(\widetilde{\Lambda},\eta)$. Taking
supremum on $\mathcal{M}(\widetilde{\Lambda},\eta)$ and letting
$k\rightarrow +\infty$, we conclude that
$$\lim_{\theta_{k'}\rightarrow
0}h_{\mathcal{M}(\widetilde{\Lambda},\eta),\loc}^{Kat}(f,\theta_{k'}\mid
\delta)=0.$$ \end{proof}

\begin{Rem}\label{lose control of local entropy}
In \cite{Newhouse}, Lemma \ref{asymp entropy expansive} was used to
prove upper semi-continuity of metric entropy on
$\mathcal{M}(\widetilde{\Lambda},\eta)$. However, the upper
semi-continuity is broadly not true even if the underlying system is
non uniformly hyperbolic. For example, in
\cite{Downarowicz-Newhouse}, T. Downarowicz and S. E. Newhouse
established surface diffeomorphisms whose local entropy of arbitrary
pre-assigned scale is always larger than a positive constant.
Exactly, they constructed a compact subset $E$ of
$\mathcal{M}_{inv}( \Lambda, f )$ such that there exist a periodic
measure in $E$ and  a positive real number $\rho_0$ such that for
each $\mu\in E$ and each $k > 0$,
$$\limsup_{\nu\in E, \nu\rightarrow \mu} h_{\nu}( f )- h_k(\nu) >
\rho_0,$$ which implies infinity of symbolic extension entropy and
also the absence of upper semi-continuity of metric entropy and thus
no uniform separation in \cite{PS}.

\end{Rem}

Now we begin to prove Theorem \ref{main theorem of set1} and hence
complete the proof of Theorem \ref{main theorem of set} by
Proposition \ref{generical lemmas}. Throughout this section, for
simplicity, we adopt the symbols used in the proof of Theorem
\ref{main theorem of measure}. Except specially mentioned, the
relative quantitative relation of symbols share the same meaning.

{\bf Proof of Theorem \ref{main theorem of set1}}$ $

 For any
nonempty closed connected set $K\subset
\mathcal{M}(\widetilde{\Lambda},\eta)$, there exists a sequence of
closed balls $U_n$ in $\mathcal{M}_{inv}(M,f)$ with radius $\zeta_n$
in the metric $D$ with the weak$^*$ topology such that the following
holds:
\begin{eqnarray*}
&(i)&U_n\cap U_{n+1}\cap K \neq \emptyset;\\[2mm]
&(ii)& \cap_{N\geq 1}^{\infty}\cup_{n\geq N}U_n =
K;\\[2mm]
&(iii)& \lim_{n\rightarrow +\infty}\zeta_n=0.\end{eqnarray*} By
$(1)$, we take $\nu_k\in U_k\cap K$. Given $\gamma>0$, using Lemma
\ref{asymp entropy expansive}, we can find an $\varepsilon>0$ such
that
$$h_{\mathcal{M}(\widetilde{\Lambda},\eta),\loc}^{Kat}(f,\varepsilon\mid
\delta)<\gamma.$$ For each $\nu_k$, we then can choose a finite
convex combination of ergodic probability measures with rational
coefficients,
$$\mu_k=\sum_{j=1}^{p_k}a_{k,j}\,m_{k,j}$$
satisfying  the following properties:
$$D(\nu_k,\mu_k)<\frac{1}{k},\,\,\,m_{k,j}(\Lambda)=1\,\,\,\mbox{
and}\,\,\, |h_{\nu_k}^{Kat}(f,\varepsilon\mid
\delta)-h_{\mu_k}^{Kat}(f,\varepsilon\mid \delta)|<\frac1k.$$ For
each $k$, we can find $l_k$ such that
$m_{k,j}(\Lambda_{l_k})>1-\delta$ for all $1\leq j\leq p_k$. For
$c=\frac{\varepsilon}{8\epsilon_0}$, by Theorem \ref{specification}
there is a sequence of numbers $(\delta_k)_{k=1}^\infty$. Let
$\xi_k$ be a finite partition of $M$ with $\mbox{diam}\,\xi_k
<\min\{\frac{b_k(1-e^{-\epsilon})}{4\sqrt{2}e^{(k+1)\epsilon}},\epsilon_{l_k},\delta_{l_k}\}$
and $\xi_k>\{\widetilde{\Lambda}_{l_k},M\setminus
\widetilde{\Lambda}_{l_k}\}$.

For each $m_{k,j}$, following the proof of Theorem \ref{main theorem
of measure1}, we can obtain an integer $n(k,j)$ and an
$(n(k,j),\varepsilon)$-separated set $W_n$ contained in an open
subset $U(k,j)$ of some Lyapunov neighborhood with
$\diam(U(k,j))<2\diam(\xi_k)$ and satisfying that
$$\sharp\,W_{n(k,j)}\geq  e^{n(k,j)(1-\frac1k)(h_{m_{k,j}}^{Kat}(f,\varepsilon\mid
\delta)-\frac4k)}.$$ Then likewise, for $k_1,k_2,j_1,j_2$ one  can
find $y=y(m_{k_1,j_1},m_{k_2,j_2})\in U(k_1,j_1)$ satisfying that
for some $s=s(m_{k_1,j_1},m_{k_2,j_2})\in\mathbb{ N}$,
$$f^{s}(y)\in U(k_2,j_2).$$In the same manner, we consider the
following pseudo-orbit
\begin{eqnarray}
 \label{quasi orbits}\quad  \quad O&=&O(x(1,1;1,1),\cdots,x(1,1; 1, N_1C_{1,1}),\cdots,x(1,p_1; 1, 1),\cdots,x(1,p_1; 1, N_1C_{1,p_1});\nonumber\\[2mm]&&\cdots;\nonumber \\[2mm]
 &&x(1,1;T_1,1),\cdots,x(1,1; T_1, N_1C_{1,1}),\cdots,x(1,p_1; T_11),\cdots,x(1,p_1; T_1, N_1C_{1,p_1});\nonumber\\[2mm]&&\cdots;\nonumber \\[2mm]
 &&x(k,1; 1, 1),\cdots,x(k,1; 1, N_kC_{k,1}),\cdots,x(k,p_k,; 1, 1),\cdots,x(k,p_k; 1, N_kC_{k,p_k});\nonumber\\[2mm]&&\cdots;\nonumber \\[2mm]
 &&x(k,1;T_k,1),\cdots,x(K, 1; T_k, N_kC_{k,1}),\cdots,x(k,p_k; T_k, 1),\cdots,x(k,p_k; T_k, N_kC_{k,p_k});\nonumber\\[2mm]&&\cdots;\nonumber \\[2mm]
 &&\cdots\,\,)\nonumber\end{eqnarray}
with the precise type as (\ref{precise quasi orbits}), where $x(k,j;
i, t)\in W_{n(k,j)}$. Then Theorem \ref{specification} applies to
give rise to a  shadowing point $z$ of $O$ such that
$$d(f^{M_{k,i,j,t}+q}(z),f^q(x(k,j;i, t)))<c\epsilon_0e^{-\epsilon l_k}<\frac{\varepsilon}{4\epsilon_0}\epsilon_0e^{-\epsilon l_k}\leq \frac{\varepsilon}{4},$$
for $0\leq q\leq n(k,j)-1,$ $1\leq i \leq T_k$, $1\leq t\leq
N_kC_{k,j}$,  $1\leq j\leq p_k$.  By the construction of $N_k$ and
$Y_k$, it is verified that
$$D(\mathcal{E}_{Y_k}(f^{M_{k,i}}(z)),\,\nu_k)\leq\frac6k.$$
For sufficiently large $M_{k,i}\leq n\leq M_{k,i+1}$, by affine property, we have that
\begin{eqnarray*}
D(\mathcal{E}_{n}(z),\,\nu_k)&\leq&
\frac{M_{k-2}}{n}D(\mathcal{E}_{M_{k-2}}(z),\nu_k)+\frac{Y_{k-1}}{n}\sum_{r=1}^{T_{k-1}}D(\mathcal{E}_{Y_{k-1}}(f^{M_{k-1,r-1}}(z)),\nu_k)\\[2mm]
&&+\frac{s(m_{k-1,1},m_{k,1})}{n}D(\mathcal{E}_{s(m_{k-1,1},m_{k,1})}(f^{M_{k-1,T_{k-1}}}(z)),\nu_k)\\[2mm]
&&+\frac{Y_k}{n}\sum_{r=1}^{i-1}D(\mathcal{E}_{Y_k}(f^{M_{k,r-1}}(z)),\nu_k)\\[2mm]
&&+\frac{n-M_{k,i}}{n}D(\mathcal{E}_{n-M_{k,i}}(f^{M_{k,i}}(z)),\nu_k).
\\[2mm]
\end{eqnarray*}
Noting that $$D(\mathcal{E}_{Y_{k-1}}(f^{M_{k-1,i-1}}(z)),\nu_k)\leq
D(\mathcal{E}_{Y_{k-1}}(f^{M_{k-1,i-1}}(z)),\nu_{k-1})+D(\nu_{k-1},\nu_{k})$$
and using the fact that
$D(\nu_k,\nu_{k-1})\leq2\zeta_k+2\zeta_{k-1}$ and inequalities
(\ref{circle1}) and (\ref{circle2}), one can deduce that
\begin{eqnarray*}
D(\mathcal{E}_{n}(z),\,\nu_k)&\leq&
\frac1k+(\frac{6}{k-1}+2\zeta_k+2\zeta_{k-1})+\frac1k+\frac6k+\frac1k.
\end{eqnarray*}
 Letting
$n\rightarrow +\infty$, we get $V(z)\subset K$. On the other hand,
noting that $$\cap_{N\geq 1}^{\infty}\cup_{n\geq N}U_n = K,$$  so
$\mathcal{E}_{n}(z)$ can enter any neighborhood of each $\nu\in K$
in infinitely times, which implies the converse side $K\subset
V(x)$. Consequently, $V(z)=K$.

Next we show the inequality concerning entropy. Fixing $k,j,i, t $,
the corresponding shadowing points of distinct $x(k,j,i)$ are
$(n(k,j),\frac{\varepsilon}{2})$-separated. Let
\begin{eqnarray*}
H_{ki}=\{&&(x(k,j;i,1),\cdots,x(k,j;i,N_kC_{k,j}),\cdots,x(k,p_k;i,1),\cdots,x(1,p_k;i, N_kC_{k,p_k})\\[2mm]
 &&\mid \,\,x(k,j;i,t)\in V_{n(k,j)}\cap A_{n(k,j)}\}.
\end{eqnarray*}
 Then
\begin{eqnarray*}
\sharp H_{ki}\geq
e^{\sum_{j=1}^{p_k}N_kC_{k,j}n(k,j)(1-\frac1k)(h_{m_{k,j}}^{Kat}(f,\varepsilon\mid
\delta)-\frac{4}{k})}.
\end{eqnarray*}
So,
\begin{eqnarray*}
\frac{1}{Y_k}\log \,H_{ki}&\geq&
\frac{Y_k-X_k}{Y_k}\sum_{j=1}^{p_k}a_{k,j}(1-\frac1k)(h_{m_{k,j}}^{Kat}(f,\varepsilon\mid\delta)-\frac{4}{k})\\[2mm]
&\geq&(1-\frac{1}{k})\sum_{j=1}^{p_k}a_{k,j}(1-\frac1k)(h_{m_{k,j}}^{Kat}(f,\varepsilon\mid\delta)-\frac{4}{k})\\[2mm]
&=&(1-\frac1k)^2h_{\mu_{k}}^{Kat}(f,\varepsilon\mid\delta)-\frac4k(1-\frac{1}{k})^2\\[2mm]
&\geq&(1-\frac1k)^2(h_{\nu_k}^{Kat}(f,\varepsilon\mid\delta)-\frac1k)-\frac4k(1-\frac{1}{k})^2\\[2mm]
&\geq&(1-\frac1k)^2(h_{\nu_k}(f)-\gamma-\frac1k)-\frac4k(1-\frac{1}{k})^2.
\end{eqnarray*}
In sequel by the analogous  arguments  in  section 4, we obtain that
$$h_{\top}(f,G_{K})\geq \inf\{h_{\mu}(f)\mid \mu\in K\}-\gamma.$$
The arbitrariness of $\gamma$ concludes the desired inequality:
$$h_{\top}(f,G_{K})\geq \inf\{h_{\mu}(f)\mid \mu\in K\}.$$\hfill$\Box$

\section{On the Structure of Pesin set $\widetilde{\Lambda}$}

The construction of $\widetilde{\Lambda}$ asks for many techniques
that yields fruitful properties of Pesin set but meanwhile leads
difficulty to check which measures support on $\widetilde{\Lambda}$.
Sometimes $\mathcal{M}_{inv}(\widetilde{\Lambda}, f)$ contains only
the measure $\omega$ itself, for instance $\omega$ is atomic.  In
what follows, we will show that for several classes of
diffeomorphisms derived from Anosov systems
$\mathcal{M}_{inv}(\widetilde{\Lambda},f)$ enjoys many members.

\subsection{Symbolic dynamics of Anosov
diffeomorpisms}\quad Let $f_0$ be an  Anosov diffeomorphism on a
Riemannian manifold $M$. For $x\in M$, $\varepsilon_0>0$, we have
the stable manifold $W^s_{\varepsilon_0}(x)$ and the unstable
manifold $W^u_{\varepsilon_0}(x)$ defined by
\begin{eqnarray*}&&W^s_{\varepsilon_0}(x)=\{y\in M\mid
d(f_0^n(x),f_0^n(y))\leq \varepsilon_0,\quad \mbox{for all}\quad
n\geq
0\}\\[2mm]
&&W^u_{\varepsilon_0}(x)=\{y\in M\mid d(f_0^{-n}(x),F^{-n}(y))\leq
\varepsilon_0,\quad \mbox{for all}\quad n\geq 0\}.\end{eqnarray*}
 Fixing small $\varepsilon_0>0$ there exists a
$\delta_0>0$ so that $W^s_{\varepsilon_0}(x)\cap
W^u_{\varepsilon_0}(y)$ contains a single point $[x,y]$ whenever
$d(x,y)<\delta_0$. Furthermore, the function
$$[\cdot, \cdot]: \{(x,y)\in M\times M\mid d(x,y)<\delta_0 \}\rightarrow M$$
is continuous.  A  rectangle $R$ is understood by a subset of $M$
with small diameter and $[x,y]\in R$ whenever $x,y\in R$. For $x\in
R$ let $$W^s(x,R)=W^s_{\varepsilon_0}(x)\cap R\quad \mbox{and}\quad
W^u(x,R)=W^u_{\varepsilon_0}(x)\cap R.$$ For Anosov diffeomorphism
$f_0$ one can obtain the follow structure known as a  Markov
partition $\mathcal{R}=\{R_1,R_2,\cdots,R_l\}$ of $M$ with
properties:
\begin{enumerate}
\item[(1)] $\Int R_i\cap \Int R_j=\emptyset$ for $i\neq j$;

\item[(2)] $f_0 W^u(x,R_i)\supset W^u(f_0x,R_j)$ and \\
$f_0 W^s(x,R_i)\subset W^s(f_0x,R_j)$ when $x\in \Int R_i$, $fx\in
\Int R_j$.

\end{enumerate}
Using the Markov Partition $\mathcal{R}$ we can define the
transition matrix $B=B(\mathcal{R})$ by

$$B_{i,j}=\begin{cases}1\quad \mbox{if}\quad  \Int R_i\cap f_0^{-1} ( \Int R_j)\neq \emptyset;\\ 0\quad \mbox{otherwise}.
\end{cases}$$
The subshift $(\Sigma_B,\sigma)$ associated with $B$ is given by
$$\Sigma_B=\{\underline{q}\in \Sigma_l\mid \,B_{q_iq_{i+1}}=1\quad \forall i\in \mathbb{Z}\}.$$
For each $\underline{q}\in \Sigma_B$ by the hyperbolic property the
set $\cap_{i\in \mathbb{Z}}f_0^{-i}R_{q_i}$ contains of a single
point, denoted by $\pi_0( \underline{q} )$.  We denote
$$\Sigma_B(i)=\{\underline{q}\in \Sigma_B\mid q_0=i\}.$$
The following properties hold for the map $\pi_0$ (see Sinai
\cite{Sinai} and Bowen \cite{Bowen1, Bowen2}).
\begin{Prop}\label{property of Markov}$ $\\
\begin{enumerate}
\item[(1)] The map $\pi_0: \Sigma_B\rightarrow M$ is a continuous surjection
satisfying $\pi_0\circ \sigma=f_0\circ \pi_0;$\\
\item[(2)] $\pi_0(\Sigma_B(i))=R_i$,\quad $1\leq i\leq l$;\\
\item[(3)] $h_{\top}(\sigma, \Sigma_B)=h_{\top}(f_0,M)$.
\end{enumerate}

\end{Prop}
Since $B$ is $(0,1)$-matrix, using Perron Frobenius Theorem the
maximal eigenvalue $\lambda$ of $B$ is positive and simple.
$\lambda$ has the row  eigenvector  $u=(u_1,\cdots,u_l)$, $u_i>0$,
and the column eigenvector $v=(v_1,\cdots, v_l)^{T}$, $v_i>0$. We
assume $\sum_{i=1}^lu_iv_i=1$ and denote
$(p_1,\cdots,p_l)=(u_1v_1,\cdots,u_lv_l)$. Define a new matrix
$$\mathcal{P}=(p_{ij})_{l\times l},\quad\quad\mbox{where}\quad p_{ij}=\frac{B_{ij}\,v_j}{\lambda\,v_i}.$$ Then $\mathcal{P}$
can define a Markov chain with probability $\mu_0$ satisfying
$$\mu_0([a_0a_1\cdots a_i])=p_{a_0}p_{a_0a_1}\cdots p_{a_{i-1}a_i}. $$
Then $\mu_0$ is $\sigma$-invariant and Gurevich \cite{Gu1,Gu2}
proved that $\mu_0$ is the unique maximal measure of
$(\Sigma_B,\sigma)$, that is,
$$h_{\top}(\sigma, \Sigma_B)=h_{\mu_0}(\sigma,\Sigma_B)=\log\lambda.$$
In addition, Bowen \cite{Bowen1} proved that $\pi_{0*}(\mu_0)$ is
the unique maximal measure of $f_0$ and $\pi_{0*}(\mu)(\partial
\mathcal{R})=0$, where $\partial \mathcal{R}$ consists of all
boundaries of $R_i$, $1\leq i\leq l$.

Denote $\mu_1= \pi_{0*}(\mu_0)$. Then $\mu_1(\pi_0\Sigma_B(i))=p_i$
for $1\leq i\leq l$. For $0<\gamma<1$, $N\in \mathbb{N}$ define

\begin{eqnarray*}\Gamma_N(i,\gamma)=\{x\in M\mid&&\sharp \{n\leq j\leq n+k-1\mid
f_0^j(x)\in R_i\}\leq N+k(p_i+\gamma)+|n|\gamma,\\[2mm]&&\sharp \{n\leq j\leq n+k-1\mid
f_0^{-j}(x)\in R_i\}\leq N+k(p_i+\gamma)+|n|\gamma\\[2mm]&&\quad \forall \,k\geq
1,\,\,\,\forall\, n\in \mathbb{Z}\}.\end{eqnarray*}  Then
$f_0^{\pm}(\Gamma_N(i,\gamma))\subset \Gamma_{N+1}(i,\gamma)$. Let
$\Gamma(i,\gamma)=\cup_{N\geq 1}\Gamma_N(i,\gamma)$.
\begin{Lem}\label{full measure}
For any $m\in \mathcal{M}_{inv}(M,f_0)$, if \,$m(R_i)<p_i+\gamma/2$
then $m(\Gamma(i,\gamma))=1$. \end{Lem}
\begin{proof}
Since $m(R_i)<p_i+\gamma/2$, for $m$ almost all $x$ one can fine
$N(x)>0$ such that
\begin{eqnarray*}&&n(m(R_i)-\frac\gamma2)\leq\sharp \{0\leq j\leq
n-1\mid
f_0^j(x)\in R_i\}\leq n(m(R_i)+\frac\gamma2),\quad \forall \,n\geq N(x);\\[2mm]
&&n(m(R_i)-\frac\gamma2)\leq\sharp \{0\leq j\leq n-1\mid
f_0^{-j}(x)\in R_i\}\leq n(m(R_i)+\frac\gamma2),\quad \forall
\,n\geq N(x).
\end{eqnarray*}
Take $N_0(x)$ to be the smallest number such that for every $n\geq
1$,
\begin{eqnarray*}&&-N_0(x)+n(m(R_i)-\frac\gamma2)\leq\sharp \{0\leq j\leq n-1\mid
f_0^j(x)\in R_i\}\leq N_0(x)+n(m(R_i)+\frac\gamma2);\\[2mm]
&&-N_0(x)+n(m(R_i)-\frac\gamma2)\leq\sharp \{0\leq j\leq n-1\mid
f_0^{-j}(x)\in R_i\}\leq N_0(x)+n(m(R_i)+\frac\gamma2).
\end{eqnarray*}
Then for any $k\geq 1$,
\begin{eqnarray*}
&&\sharp \{n\leq j\leq n+k-1\mid f_0^j(x)\in R_i\}\\[2mm]
&=& \sharp
\{0\leq j\leq n+k-1\mid f_0^j(x)\in R_i\}-\sharp \{0\leq j\leq
n-1\mid
f_0^j(x)\in R_i\}\\[2mm]
&\leq&
N_0(x)+(n+k)(m(R_i)+\frac\gamma2)-(-N_0(x)+n(m(R_i)-\frac\gamma2))\\[2mm]
&=& 2N_0(x)+k(m(R_i)+\frac\gamma2)+n\gamma.
\end{eqnarray*}
In this manner we can also show $$\sharp \{n\leq j\leq n+k-1\mid
f_0^j(x)\in R_i\}\leq 2N_0(x)+k(m(R_i)+\frac\gamma2)+n\gamma.$$
Thus, $x\in \Gamma_{N_0(x)}(i,\gamma)$.
\end{proof}

By Lemma \ref{full measure}, $\mu_1(\Gamma(i,\gamma))=1$. We further
define
$$\widetilde{\Gamma}_N(i,\gamma)=\supp(\mu_1\mid
\Gamma_N(i,\gamma))\quad \mbox{ and}\quad
\widetilde{\Gamma}(i,\gamma)=\cup_{N\geq
1}\widetilde{\Gamma}_N(i,\gamma).$$ It holds that
$\widetilde{\Gamma}(i,\gamma)$ is $f$-invariant and
$\mu_1(\widetilde{\Gamma}(i,\gamma))=1$.

\begin{Prop}\label{frequence} There is a neighborhood $U$ of $\mu_1$ in
$\mathcal{M}_{inv}(M,f_0)$  such that for any ergodic measure $m\in
U$ we have $m \in
\mathcal{M}_{inv}(\widetilde{\Gamma}(i,\gamma),f_0)$.

\end{Prop}
\begin{proof}
Observing that $\mu_1(\partial R_i)=0$, for $\gamma>0$ there exists
a neighborhood $U$ of $\mu_1$ in $\mathcal{M}_{inv}(M,F)$ such that
for any $m\in U$ one has
$$m( R_i)<p_i+\frac\gamma2.$$
{\bf Claim :} \,\,\,\,We can find an ergodic measure $m_0\in
\mathcal{M}_{inv}(\Sigma_B, \sigma)$ satisfying $\pi_{0*}m_0=m$.

\noindent{\it Proof of Claim.} Denote the basin of $m$ by
\begin{eqnarray*}Q_{m}(M,f_0)=\Big{\{}x\in M\mid \lim_{n\to+\infty}\frac 1n
\sum_{j=0}^{n-1}\varphi(f_0^ix) &=&\lim_{n\to-\infty}\frac 1n
\sum_{j=0}^{n-1}\varphi(f_0^ix)\\[2mm]&=&\int_M\varphi dm,\quad\forall
\varphi\in C^0(M)\Big{\}}.\end{eqnarray*} Take and fix a point $x\in
Q_{m}(M,\, f_0)$ and choose $q\in \Sigma_B$ with $\pi_0(q)=x$.
Define a sequence of measures $\nu_n$ on $\Sigma_B$ by
$$\int\,\psi d \nu_n:=\frac 1n \Sigma_{i=0}^{n-1}\psi(\sigma^i(q)),\,\,\,\,
\,\,\forall \psi\in C^0(\Sigma_B).$$ By taking a subsequence when
necessary we can assume that $\nu_n\to \nu_0.$ It is standard to
verify that $\nu_0$ is a $\sigma$-invariant measure and $\nu_0$
covers $m$ i.e., $\pi_{0*}(\nu_0)=m.$ Set
$$Q(\sigma):=
\cup_{\nu \in \mathcal{M}_{erg}(\Sigma_B, \sigma)} Q_\nu(\Sigma_B,
\sigma).$$ Then $ Q(\sigma)$ is a $\sigma-$invariant total measure
subset in $\Sigma_B.$ We have
\begin{align*}
&m( Q_{m}(M,F) \cap \pi_{0} Q(\sigma))\\
\ge & \nu_0(\pi_{0}^{-1}Q_{m}(M,f_0)\cap Q(\sigma))\\
=&1.
\end{align*}
Then the set
\begin{align*} \mathcal{A}_0:=\Big{\{}\nu\in \mathcal{M}_{erg}(\Sigma_B, \sigma)\,\mid\,\,&\exists \,\, q\in
Q(\sigma), \pi_0(q)\in Q_{m}(M, f_0), s.\,t. \\
&\lim_{n\to+\infty}\frac 1n\Sigma_{i=0}^{n-1} \psi(\sigma^i(q))
=\lim_{n\to-\infty}\frac 1n\Sigma_{i=0}^{n-1}
\psi(\sigma^i(q))\\
&=\int _{\Sigma_B} \psi\,d\nu\,\,\,\,\,\,\quad\quad\quad \forall
\psi\in C^0(\Sigma_B)\,\Big{\}}
\end{align*}
is non-empty. It is clear that $\nu $ covers $m$, $\pi_{0*}(\nu)=m,$
for all $\nu\in \mathcal{A}_0.$

\hfill $\Box$

We continue the proof of Proposition \ref{frequence}. Since
$\pi_{0*}(m_0)=m$ so $m_0(\pi_{0}^{-1}(R_i))=m(R_i)< p_i+\gamma/2$
which together with $\Sigma_B(i)\subset \pi_0^{-1}(R_i)$ implies
that
$$m_0(\Sigma_B(i))< p_i+\frac\gamma2.$$
In particular, $\mu_0(\Sigma_B(i))< p_i+\frac\gamma2.$ For
$0<\gamma<1$, $N\in \mathbb{N}$ define
\begin{eqnarray*}\Upsilon_N(i,\gamma)=\{\underline{q}\in \Sigma_B\mid &&\sharp
\{n\leq j\leq n+k-1\mid q_j=i\}\leq  N+k(p_i+\gamma)+|n|\gamma,\\[2mm]
&& \sharp
\{n\leq j\leq n+k-1\mid q_{-j}=i\}\leq N+k(p_i+\gamma)+|n|\gamma\\[2mm]
&&\quad \forall\,k\geq1\,\,\,\forall \,n\in
\mathbb{Z}\}.\end{eqnarray*} Let $\Upsilon(i,\gamma)=\cup_{N\geq
1}\Upsilon_N(i,\gamma)$. Then $\mu_0(\Upsilon(i,\gamma))=1$. Further
define
$$\widetilde{\Upsilon}_N(i,\gamma)=\supp(\mu_0\mid
\Upsilon_N(i,\gamma))\quad \mbox{ and}\quad
\widetilde{\Upsilon}(i,\gamma)=\cup_{N\geq
1}\widetilde{\Upsilon}_N(i,\gamma).$$ It also holds that
$\widetilde{\Upsilon}(i,\gamma)$ is $\sigma$-invariant and
$\mu_0(\widetilde{\Upsilon}(i,\gamma))=1$.

\begin{Lem}\label{measures on support}Given $m_0\in \mathcal{M}_{erg}(\Sigma_B,\sigma)$,
if $m_0(\Sigma_B(i))<p_i+\gamma/2$ then $m_0\in
\mathcal{M}_{inv}(\widetilde{\Upsilon}(i,\gamma),\sigma)$.

\end{Lem}

\noindent{\it Proof of Lemma.}\,\,\,\, Since
$m_0(\Sigma_B(i))<p_i+\gamma/2$ we obtain $m_0(\cup_{N\in
\mathbb{N}}\,\Upsilon_{N}(i,\gamma))=1$. We can take $N_0$ so large
that $m_0(\Upsilon_{N_0}(i,\gamma))>0$ and
$\mu_0(\widetilde{\Upsilon}_{N_0}(i,\gamma))>0$. Define
$$\Upsilon(i,j)=\{\underline{q}\in  \widetilde{\Upsilon}_{N_0}(i,\gamma)\mid q_0=j\}.$$
Then there exists $j\in [ 1,l ]$ such that $\mu_0(\Upsilon(i,j))>0$.

Noting that $(\Sigma_B,\sigma)$ is mixing, there is $L_0>0$ such
that for each pair $j_1,j_2$ one can choose an sequence
$L(j_1,j_2)=(q_1 \cdots q_L)$ satisfying $q_1=j_1$, $q_L=j_2$ and
$2\leq \sharp L(j_1,j_2)\leq L_0$.

Arbitrarily taking $\underline{q}\in \Upsilon_{N_0}(i,\gamma)$,
 $\underline{z}\in \Upsilon(i,j)$, $n\in \mathbb{N}$,  define
$$\underline{y}(\underline{q},\underline{z},n)
=(\cdots z_{-3}z_{-2}z_{-1}L(z_0,q_{-n})q_{-n+1}\cdots q_{-1};
\stackrel{0}{q_0} q_1\cdots q_{n-1}L(q_n,z_0)z_1z_2z_3 \cdots ).$$
Denote $N_1=2L_0+2N_0+1$.   For any $\theta>0$ we can take large $n$
satisfying $n>N_1$ and
$d(\underline{y}(\underline{q},\underline{z},n),\underline{q})<\theta$.
Define a new subset of $\Sigma_B$:
$$Y(\underline{q},n)=\{\underline{y}(\underline{q},\underline{z},n)\in \Sigma_B \mid
\underline{z}\in \Upsilon(i,j)\}.$$ Consider the positive and
negative constitutions of $\Upsilon(i,j)$ as follows
\begin{eqnarray*}
\Upsilon^{+}(i,j)&=&\{\underline{w}\in
\Sigma_B\mid\,w_k=z_k,\,\,i\geq
0,\quad \mbox{for some }\,\,\underline{z}\in \Upsilon(i,j)\}\\[2mm]
\Upsilon^{-}(i,j)&=&\{\underline{w}\in
\Sigma_B\mid\,w_k=z_k,\,\,i\leq 0,\quad \mbox{for some
}\,\,\underline{z}\in \Upsilon(i,j)\}.
\end{eqnarray*}
Clearly $\Upsilon^{+}(i,j)\supset \Upsilon(i,j)$,
$\Upsilon^{-}(i,j)\supset \Upsilon(i,j)$.  Then  by the Markov
property of $\mu_0$ it holds that
$$\mu_0(Y(\underline{q},n))\geq
\mu_0(\Upsilon^{-}(i,j))p_{jq_{-n}}p_{q_{-n}q_{-n+1}}\cdots
p_{q_{n-1}q_{n}}q_{q_{n}j}\,\mu_0(\Upsilon^{+}(i,j))>0.$$ Moreover,
for any $\underline{y}\in Y(\underline{q},n)$ and $k\geq1,\,\,s\in
\mathbb{Z}$ we have

Case 1: $-n-\sharp L\leq s\leq n+\sharp L$, $s+k-1\leq n+\sharp L$
it follows that
\begin{eqnarray*}\,\sharp \{s\leq t\leq s+k-1\mid
y_t=i\}&\leq&\,2L_0+\sharp \{s\leq t\leq s+k-1\mid q_t=i\} \\[2mm]&\leq&
2L_0+N_0+k(p_i+\gamma)+|s|\gamma.\end{eqnarray*}

Case 2:  $-n-L\leq s\leq n+L$, $s+k-1> n+L$ it follows that
\begin{eqnarray*}\,\sharp \{s\leq t\leq s+k-1\mid
y_t=i\}&\leq&\,L_0+N_0+(n+L-s)(p_i+\gamma)+s|\gamma|+N_0+\\[2mm]
&&+(s+k-1-n-L)(p_i+\gamma) \\[2mm]&\leq&
L_0+2N_0+k(p_i+\gamma)+|s|\gamma.\end{eqnarray*}

Case 3: $s>n+L$ it follows that
\begin{eqnarray*}\,\sharp \{s\leq t\leq s+k-1\mid
y_t=i\}&\leq & N_0+k(p_i+\gamma)+|s|\gamma.\end{eqnarray*}

Case 4: $s<-n-L$ it follows that

\begin{eqnarray*}\,\sharp \{s\leq t\leq s+k-1\mid
y_t=i\}&\leq&\,2L_0+2N_0+k(p_i+\gamma)+|s|\gamma.\end{eqnarray*} The
situation of $\sharp \{s\leq t\leq s+k-1\mid y_{-t}=i\}$ is similar.
Altogether, since $N_1=2L_0+2N_0+1$, $Y(\underline{q},n)\subset
\Upsilon_{N_1}(i,\gamma) $. The arbitrariness of $\theta$ gives rise
to that
$$\underline{q}\in \supp(\mu_0\mid \Upsilon_{N_1}(i,\gamma)).$$ That
is, $\Upsilon_{N_0}(i,\gamma)\subset
\widetilde{\Upsilon}_{N_1}(i,\gamma) $. Since
$m_0(\Upsilon_{N_0}(i,\gamma))>0$ so
$m_0(\widetilde{\Upsilon}_{N_1}(i,\gamma))>0$ which by the
ergodicity of $m_0$ implies $m_0(\widetilde{\Upsilon}(i,\gamma))=1$.

\hfill $\Box$

Noting that $\pi_0(\widetilde{\Upsilon}_{N}(i,\gamma))\subset
\widetilde{\Gamma}_{N}(i,\gamma)$, by Lemma \ref{measures on
support} we obtain
$$m(\widetilde{\Gamma}(i,\gamma))=m_0(\pi_0^{-1}(\widetilde{\Gamma}(i,\gamma)))\geq
m_0(\widetilde{\Upsilon}(i,\gamma))=1
$$ which concludes Proposition \ref{frequence}.

 \end{proof}

\subsection{Nonuniformly hyperbolic systems} We shall verify $\widetilde{\Lambda}$
for an example due to Katok \cite{Katok} (see also \cite{Ba-Pesin1,
Barr-Pesin}) of a diffeomorphism on the 2-torus $\mathbb{T}^2$ with
nonzero Lyapunov exponents, which is not an Anosov map.  Let $f_0$
be a hyperbolic linear automorphism given by the matrix

$$A=\begin{pmatrix} 2&1\\1&1
\end{pmatrix}
$$
Let $\mathcal{R}=\{R_1,R_2,\cdots,R_l\}$ be the Markov partition of
$f_0$ and $B=B(\mathcal{R})$ be the associated transition matrix.
$f_0$ has a maximal measure $\mu_1$.  Without loss of generality, at
most taking an iteration of $f_0$ we suppose there is a fixed point
$O\in \Int R_1$. Consider the disk $D_r$ centered at $O$ of radius
$r$. Let $(s_1, s_2)$ be the coordinates in $D_r$ obtained from the
eigendirections of $A$. The map $A$ is the time-1 map of the local
flow in $D_r$ generated by the system of ordinary differential
equations:
$$\frac{ds_1}{dt} = s_1 \log\lambda ,\quad \frac{ds_2}{dt} = -s_2 \log\lambda.$$ We obtain
the desired map by slowing down $A$ near the origin.

Fix small $r_1 < r_0$ and consider the time-1 map $g$ generated by
the system of ordinary differential equations in $D_{r_1}$ :

$$ \frac{ds_1}{dt} =
s_1\psi(s_1^2 + s_2^ 2) \log\lambda ,\quad \frac{ds_2}{dt}
=-s_2\psi(s_1^2 + s_2^2) \log\lambda  $$ where $\psi$ is a
real-valued function on $[0, 1]$ satisfying: \begin{enumerate}
\item[(1)] $\psi$ is a $C^{\infty}$ function except for
the origin $O$; \\
\item[(2)] $ \psi(0) = 0$ and $\psi(u) = 1$ for $u\geq r_0$ where $0 < r_0 < 1;$\\
\item[(3)] $ \psi(u) >
0$ for every $0<u<r_0$;\\
\item[(4)] $\int_0^1\frac{du}{\psi(u)}<\infty$.
\end{enumerate}
The map $f$, given as $f(x) = g(x)$ if $x\in D_{r_1}$ and $f(x) =
A(x)$ otherwise, defines a homeomorphism of the torus, which is a
$C^{\infty}$ diffeomorphism everywhere except for the origin $O$. To
provide the differentiability of the map $f$, the function $\psi$
must satisfy some extra conditions. Namely, near $O$ the integral
$\int_0^1du/\psi$ must converge ``very slowly". We refer the
smoothness to \cite{Katok}. Here $f$ is contained in the $C^0$
closure of Anosov diffeomorphisms and even more there is a
homeomorphism $\pi: \mathbb{T}^2 \rightarrow \mathbb{T}^2$ such that
$\pi\circ f_0=f\circ \pi$ and $\pi(O)=O$. By the constructions,
there is a continuous decomposition on the tangent space
$T\mathbb{T}^2=E^1\oplus E^2$ such that
 for any neighborhood $V$ of $O$, there exists
$\lambda_V>1$ such that
\begin{enumerate}
\item[(1)]$\|Df_x \mid_{ E^1(x)}\|\geq \lambda_{V}$, \,\,\, $\|Df_x \mid_{ E^2(x)}\|\leq
\lambda_{V}^{-1}$,\,\,\, $x\in \mathbb{T}^2\setminus V$ ;\\
\item[(2)] $\|Df_x \mid_{ E^1(x)}\|\geq 1$, \,\,\, $\|Df_x \mid_{ E^2(x)}\|\leq
1$,\,\,\, $x\in  V$.
\end{enumerate}
Let $H_i=\pi(R_i)$, $\nu_0=\pi_{*}\mu_1$ and $p_i=\nu_0(H_i)$. Then
$H_i$ is a closed subset of $\mathbb{T}^2$ with nonempty interior.
Let
\begin{eqnarray*}p_0&=&\frac12\min\{1-p_i\,\mid\,\, 1\leq i\leq
l\},\\[2mm]
\beta&=&(1-p_1-p_0-\gamma)\log\lambda_{V}.
\end{eqnarray*}

\begin{Thm} \label{Katok}There exists a
neighborhood $U$ of $\nu_0$ in $\mathcal{M}_{inv}(\mathbb{T}^2,f)$
such that for any ergodic $\nu\in U$ it holds that $\nu\in
\mathcal{M}_{inv}(\widetilde{\Lambda}(\beta,\beta, \epsilon))$ for
any $0\leq \epsilon\ll \beta$.

\end{Thm}

\begin{proof}Take a small neighborhood $V\subset H_1$ of $O$.  Denote
\begin{eqnarray*}\Phi_N(i,\gamma)=\{x\in M\mid&&\sharp \{n\leq j\leq n+k-1\mid
f^j(x)\in H_i\}\leq N+k(p_i+\gamma)+|n|\gamma,\\[2mm]&&\sharp \{n\leq j\leq n+k-1\mid
f^{-j}(x)\in H_i\}\leq N+k(p_i+\gamma)+|n|\gamma\\[2mm]&&\quad \forall \,k\geq
1,\,\,\,\forall\, n\in \mathbb{Z\}}.\end{eqnarray*} Define
\begin{eqnarray*}\widetilde{\Phi}_N(1,\gamma)&=&\supp(\nu_0\mid
\Phi_N(1,\gamma)).
\end{eqnarray*} Then for some large $N$ we have $\nu_0(\widetilde{\Phi}_N(1,\gamma))>0$. Noting that $\mu_1(\partial R_1)=0$, by Proposition \ref{frequence}
and the conjugation $\pi$ there exists a neighborhood $U$ of $\nu_0$
in $\mathcal{M}(\mathbb{T}^2,f)$ such that for any ergodic $\nu\in
U$,
$$\nu(\widetilde{\Phi}_N(1,\gamma))>0.$$
 For any $x\in
\Phi_N(1,\gamma)$ and $k\geq 1$, $n\in \mathbb{Z}$ we have

Case 1: $k(p_1+\gamma+p_0)\leq N+k(p_1+\gamma)+|n|\gamma$, then
$$k\leq \frac{N+|n|\gamma}{p_0}.$$
So, \begin{eqnarray*}
\|Df^{-k} \mid_{ E^1(f^{n}x)}\|&\leq& e^{-k\beta}\exp(\frac{\beta}{p_0}(N+|n|\gamma)),\\[2mm]
\|Df^k_x
\mid_{ E^2(f^nx)}\|&\leq&
e^{-k\beta}\exp(\frac{\beta}{p_0}(N+|n|\gamma)).
\end{eqnarray*}

Case 2: $k(p_1+\gamma+p_0)> N+k(p_1+\gamma)+|n|\gamma$, then
\begin{eqnarray*}
\|Df^{-k} \mid_{ E^1(f^{n}x)}\|&\leq& \lambda_V^{-(1-p_1-p_0-\gamma)k}=e^{-\beta k},\\[2mm] \|Df^k_x
\mid_{ E^2(f^nx)}\|&\leq& \lambda_V^{-(1-p_1-p_0-\gamma)k}=e^{-\beta
k}.
\end{eqnarray*}
Let $N_2=[\frac{\beta N}{\gamma p_0}]+1$. Then
$$\Phi_N(1,\gamma)\subset \Lambda_{N_2}(\beta,\beta, \gamma)
\quad \mbox{and}\quad \widetilde{\Phi}_N(1,\gamma)\subset
\widetilde{\Lambda}_{N_2}(\beta,\beta, \gamma).$$
 Therefore,
$$\nu(\widetilde{\Lambda}_{N_2}(\beta,\beta, \gamma))>0$$
which completes the proof of Theorem \ref{Katok}.

\end{proof}

\subsection{Robustly transitive partially hyperbolic systems}
In \cite{Mane} R.~Ma{\~{n}}{\'{e}} constructed a class of robustly
transitive diffeomorphisms which is not hyperbolic. Firstly we
recall the description of Ma{\~{n}}{\'{e}}'s example. Let
$\mathbb{T}^n$, $n\geq 3$, be the torus $n$-dimentional and $f_0 :
\mathbb{T}^n \rightarrow \mathbb{T}^n$ be a (linear) Anosov
diffeomorphism. Assume that the tangent bundle of $\mathbb{T}^n$
admits the $Df_0$-invariant splitting $T\mathbb{T}^n = E^{ss} \oplus
E^u\oplus E^{uu}$, with $\dim E^u = 1$ and $$\lambda_s :=
\|Df\mid_{E^{ss}}\|,\quad \lambda_u:= \|Df\mid_{ E^u}\|,\quad
\lambda_{uu}:= \|Df\mid_{E^{uu}}\|$$ satisfying the relation
$$\lambda_s < 1 < \lambda_u < \lambda_{uu}.$$
The following Lemma is proved in \cite{Sambarino-Vasquez}.
\begin{Lem}\label{shadowing onto diffeo}
Let $f_0 : \mathbb{T}^n\rightarrow \mathbb{T}^n$  be a linear Anosov
map. Then there exists $C> 0$ such that for any small $r$ and any $f
:  \mathbb{T}^n\rightarrow \mathbb{T}^n$ with $\dist_{C^0}(f, g) <
r$ there exists $\pi : \mathbb{T}^n\rightarrow \mathbb{T}^n$
continuous and onto, $\dist_{C^0}(\pi, \id) < Cr$, and $$f_0\circ
\pi = \pi\circ f.$$
\end{Lem}
Let $\mathcal{R}=\{R_1,R_2,\cdots,R_l\}$ be the Markov partition of
$f_0$ and $B=B(\mathcal{R})$ be the associated transition matrix.
Let $\mu_1$ be the maximal measure of $(\mathbb{T}^n,f_0)$ and
$p_i=\mu_1(R_i)$ for $1\leq i\leq l$. Suppose there is a fixed point
$O\in \Int R_1$. Take  small $r$ satisfying the ball $B(O,Cr)\subset
R_1$ and $d(B(O,Cr), \partial R_1)>Cr$.  Then deform the Anosov
diffeomorphim $f$ inside $B(p, r)$ passing through a flip
bifurcation along the central unstable foliation $\mathcal{F}^u(p)$
and then we obtain three fixed points, two of them with stability
index equal to $\dim E^s$ and the other one with stability index
equal to $\dim E^s + 1$. Moreover take positive numbers $\delta,
\gamma \ll \min\{\lambda_s, \lambda_u\}$.   Let $f$ satisfy the
following $C^1$ open conditions:

\begin{enumerate}\item[(1)] $ \|Df\mid_{E^{ss}}\|< e^{\delta}\lambda_s,\quad
\|Df\mid_{E^{uu}}\|>e^{-\delta}\lambda_{uu}$;\\
\item[(2)] $ e^{-\delta}\lambda_u<\|Df\mid_{ E^u(x)}\|< e^{\delta}\lambda_u$
,\quad for $x\in \mathbb{T}^n\setminus B(O,r)$;\\
\item[(3)] $ e^{-\delta}<\|Df\mid_{ E^u(x)}\|< e^{\delta}\lambda_u$
,\quad for $x\in B(O,r)$.
\end{enumerate}
\begin{figure}[h]
\begin{center}
\begin{picture}(100,90)(-3,-3)

\put(0,0){\line(1,1){40}}

\put(30,0){\line(1,1){40}}

\put(60,0){\line(1,1){40}}

\put(-15,20){\line(1,0){130}}

\put(27,27){\vector(1,1){0}}

\put(30,30){\vector(1,1){0}}

\put(57,27){\vector(1,1){0}}

\put(60,30){\vector(1,1){0}}

\put(87,27){\vector(1,1){0}}

\put(90,30){\vector(1,1){0}}

\put(12,12){\vector(-1,-1){0}}

\put(9,9){\vector(-1,-1){0}}

\put(39,9){\vector(-1,-1){0}}

\put(42,12){\vector(-1,-1){0}}

\put(72,12){\vector(-1,-1){0}}

\put(69,9){\vector(-1,-1){0}}

\put(65,20){\vector(-1,0){0}}

\put(35,20){\vector(1,0){0}}

\put(100,20){\vector(1,0){0}}

\put(5,20){\vector(-1,0){0}}

\put(50,20){\line(0,1){35}}

\put(80,20){\line(0,1){35}}

\put(20,20){\line(0,1){35}}

\put(50,35){\vector(0,-1){0}}

\put(80,35){\vector(0,-1){0}}

\put(20,35){\vector(0,-1){0}}

\put(60,23){\makebox(0,0){\tiny $O$}}

\put(50,20){\circle*{2}}

\end{picture}

\end{center}
\caption{}
\end{figure}

As shown in \cite{Sambarino-Vasquez} for the obtained $f$ there
exists a unique maximal measure $\nu_0$ of $f$ with
$\pi_{*}\nu_0=\mu_1$. Let $H_i=\pi(R_i)$, $p_i=\pi_{*}\mu_1(H_i)$
and
\begin{eqnarray*}p_0&=&\frac12\min\{1-p_i\,\mid\,\, 1\leq i\leq
l\},\\[2mm]
\beta&=&(1-p_1-p_0-\gamma)\min\{-\log\lambda_{s}-\delta,\,\,\log\lambda_{u}-\delta\}.
\end{eqnarray*}
We can see $E^{uu}$ is uniformly contracted by at least
$e^{-\beta}$.
\begin{Thm} \label{Mane}There exist $0<\epsilon\ll1<\beta$ and a
neighborhood $U$ of $\nu_0$ in $\mathcal{M}_{inv}(\mathbb{T}^n,f)$
such that for any ergodic $\nu\in U$ it holds that $\nu\in
\mathcal{M}_{inv}(\widetilde{\Lambda}(\beta,\beta, \epsilon), f)$.
\end{Thm}
\begin{proof}
By Proposition \ref{frequence} we can take a neighborhood $U_1$ of
$\mu_1$ in $\mathcal{M}_{inv}(\mathbb{T}^n,f_0)$ such that every
ergodic $\mu\in U_1$ also belongs to $
\widetilde{\Gamma}(i,\gamma)$, where $\widetilde{\Gamma}(i,\gamma)$
is given by Proposition \ref{frequence}.  Since
 $\pi$ is continuous, there is a neighborhood $U$ of $\nu_0$ in $\mathcal{M}_{inv}(\mathbb{T}^n,f)$ such
 that $\pi_{*}U\subset U_1$. For $N\in \mathbb{N}, \gamma>0$, define
\begin{eqnarray*}
T_N(i,\gamma)=\{x\in M\mid&&\sharp \{n\leq j\leq n+k-1\mid
f^j(x)\in B(O,r)\}\leq N+k(p_i+\gamma)+|n|\gamma,\\[2mm]&&\sharp \{n\leq j\leq n+k-1\mid
f^{-j}(x)\in B(O,r)\}\leq N+k(p_i+\gamma)+|n|\gamma\\[2mm]&&\quad \forall \,k\geq
1,\,\,\,\forall\, n\in \mathbb{Z\}}.
\end{eqnarray*}
For large $N$ we have $\nu_0(T_N(i,\gamma))>0$ and let
$$\widetilde{T}_N(i,\gamma)=\supp(\nu_0\mid T_N(i,\gamma)).$$
 For any $z\in T_N(i,\gamma)$, $n\in \mathbb{Z}$,
$k\geq 1$ we have

Case 1: $k(p_1+\gamma+p_0)\leq N+k(p_1+\gamma)+|n|\gamma$, then
$$k\leq \frac{N+|n|\gamma}{p_0}.$$
So, \begin{eqnarray*} \|Df^{-k} \mid_{ E^{u}(x)\oplus
E^{uu}(f^{n}x)}\|&\leq&
e^{-k\beta}\exp(\frac{\beta}{p_0}(N+|n|\gamma)).
\end{eqnarray*}

Case 2: $k(p_1+\gamma+p_0)> N+k(p_1+\gamma)+|n|\gamma$, then
\begin{eqnarray*}
\|Df^{-k} \mid_{E^{u}(x)\oplus E^{uu}(f^{n}x)}\|\leq
(\lambda_ue^{\delta})^{-(1-p_1-p_0-\gamma)k} e^{\delta
k(p_1+\gamma+p_0)}\leq  e^{-\beta k}e^{\delta k(p_1+\gamma+p_0)}.
\end{eqnarray*}
Let $N_2=[\frac{\beta N}{\gamma p_0}]+1$, $\epsilon= \max\{\delta
(p_1+\gamma+p_0),\,\gamma\}$. Then
$$T_N(1,\gamma)\subset \Lambda_{N_2}(\beta,\beta, \epsilon)
\quad \mbox{and}\quad \widetilde{T}_N(1,\gamma)\subset
\widetilde{\Lambda}_{N_2}(\beta,\beta, \epsilon).$$
 For any $x\in
\Gamma_N(1,\gamma)$, $z\in \pi^{-1}(x)$, it holds that
$$d(f^i(z),f^i_0(x))=d(f^i(z),\, \pi (f^i(x)))<Cr$$ which implies
that if $f^i_0(x)\notin R_1$ then $f^i(z)\notin B(O,r)$ because
$d(B(O,r),\partial R_1)>Cr$. Thus
$$\pi^{-1}(\Gamma_N(1,\gamma))\subset T_N(1,\gamma)\subset
\Lambda_{N_2}(\beta , \beta, \epsilon)$$ which yields that
$$\pi^{-1}(\widetilde{\Gamma}_N(1,\gamma))\subset
\widetilde{\Lambda}_{N_2}(\beta, \beta, \epsilon).$$ For any ergodic
$\nu\in U$, $\pi_{*}\nu\in U_1$. So
$\pi_{*}\nu(\widetilde{\Gamma}_{N}(1,\gamma))>0$. We obtain
$$\nu(\widetilde{\Lambda}_{N_2}(\beta, \beta, \epsilon))
\geq
\nu(\pi^{-1}(\widetilde{\Gamma}_N(1,\gamma)))=\pi_{*}\nu(\widetilde{\Gamma}_{N}(1,\gamma))>0.$$
The ergodicity of $\nu$ concludes $\nu(\widetilde{\Lambda}(\beta,
\beta, \epsilon))=1$.

\end{proof}

\subsection{Robustly transitive systems which is not partially hyperbolic}
 In this subsection we will apply the structure of $\widetilde{\Lambda}$  to
a class of diffeomorphisms introduced by Bonatti-Viana. For our
requirements we need do some additional assumptions  on their
constants. The class $\mathcal{V}\subset\mathrm{Diff}(\mathbb{T}^n)$
under consideration consists of diffeomorphisms which are also
deformations of an Anosov diffeomorphism. To define $\mathcal{V}$,
let $f_0$ be a linear Anosov diffeomorphism of the $n$-dimensional
torus $\mathbb{T}^n$.  Let $\mathcal{R}=\{R_1,R_2,\cdots,R_l\}$ be
the Markov partition of $f_0$ and $B=B(\mathcal{R})$ be the
associated transition matrix. Let $\mu_1$ be the maximal measure of
$(\mathbb{T}^n,f_0)$ and $p_i=\mu_1(R_i)$ for $1\leq i\leq l$.
Suppose there is a fixed point $O\in \Int R_1$. Take  small $r$
satisfying the ball $B(O,Cr)\subset R_1$ and $d(B(O,Cr), \partial
R_1)>Cr$, where $C$ is given by Lemma \ref{shadowing onto diffeo}.

Denote by $TM = E_0^s\oplus E_0^u$ the hyperbolic splitting for
$f_0$ and let
 $$\lambda_s :=
\|Df\mid_{E^{s}_0}\|,\quad \lambda_u:= \|Df\mid_{ E^u_0}\|.$$ We
suppose that $f_0$ has at least one fixed point outside $V$. Fix
positive numbers $\delta, \gamma \ll \lambda : =\min\{\lambda_s,
\lambda_u\}$. Let
\begin{eqnarray*}p_0&=&\frac12\min\{1-p_i\,\mid\,\, 1\leq i\leq
l\},\\[2mm]
\beta&=&(1-p_1-p_0-\gamma)\log\lambda-\delta,.
\end{eqnarray*}
By definition $f\in \mathcal{V}$ if it satisfies the following $C^1$
open conditions:

(1) There exist small continuous cone fields $C^{cu}$ and $C^{cs}$
invariant for $Df$ and $Df^{-1}$ containing respectively $E_0^u$ and
$E_0^s$.

(2) $f$ is $C^1$ close to $f_0$ in the complement of $B(O,r)$, so
that for $x\in \mathbb{T}^n\setminus B(O,r)$:
$$\|(Df|T_xD^{cu})\|>e^{-\delta}\lambda\,\,\,\mbox{and}\,\,\,\|Df|T_xD^{cs}\|<e^{\delta}\lambda^{-1}.$$

(3) For $x\in B(O,r)$:
$$\|(Df|T_xD^{cu})\|>e^{-\delta}\,\,\, \mbox{and}\,\,\, \|(Df|T_xD^{cs}\|<e^{\delta},$$
where $D^{cu}$ and $D^{cs}$ are disks tangent to $C^{cu}$ and
$C^{cs}$.

Immediately by the cone property, we can get a dominated splitting
$T\mathbb{T}^n=E\oplus F$ with $E\subset D^{cs}$ and $F\subset
D^{cu}$.

  Use Lemma
\ref{shadowing onto diffeo} there exists $\pi :
\mathbb{T}^n\rightarrow \mathbb{T}^n$ continuous and onto,
$\dist_{C^0}(\pi, \id) < Cr$, and
$$f_0\circ \pi = \pi\circ f.$$
In \cite{Buzzi-Fisher} for the obtained $f$,  Buzzi and Fisher
proved that there exists a unique maximal measure $\nu_0$ of $f$
with $\pi_{*}\nu_0=\mu_1$. This measure $\nu_0$ conforms good
structure of Pesin set $\widetilde{\Lambda}$ by the following
Theorem.
\begin{Thm} \label{Bonatti-Viana}There exist $0<\epsilon\ll1<\beta$ and a
neighborhood $U$ of $\nu_0$ in $\mathcal{M}_{inv}(\mathbb{T}^n,f)$
such that for any ergodic $\nu\in U$ it holds that $\nu\in
\mathcal{M}_{inv}(\widetilde{\Lambda}(\beta,\beta, \epsilon))$.
\end{Thm}
\begin{proof} The arguments are analogous of Theorem \ref{Mane}.
Choose a neighborhood $U_1$ of $\mu_1$ in
$\mathcal{M}(\mathbb{T}^n,f_0)$ such that every ergodic $\mu\in U_1$
is contained in  $ \widetilde{\Gamma}(i,\gamma)$, where
$\widetilde{\Gamma}(i,\gamma)$ defined as Proposition
\ref{frequence}. The continuity of $\pi$ give rise to a neighborhood
$U$ of $\nu_0$ in $\mathcal{M}(\mathbb{T}^n,f)$ such
 that $\pi_{*}U\subset U_1$.
For $N\in \mathbb{N}, \gamma>0$, define
\begin{eqnarray*}
T_N(i,\gamma)=\{x\in M\mid&&\sharp \{n\leq j\leq n+k-1\mid
f^j(x)\in B(O,r)\}\leq N+k(p_i+\gamma)+|n|\gamma,\\[2mm]&&\sharp \{n\leq j\leq n+k-1\mid
f^{-j}(x)\in B(O,r)\}\leq N+k(p_i+\gamma)+|n|\gamma\\[2mm]&&\quad \forall \,k\geq
1,\,\,\,\forall\, n\in \mathbb{Z\}}.
\end{eqnarray*}
For large $N$ we have $\nu_0(T_N(i,\gamma))>0$ and let
$$\widetilde{T}_N(i,\gamma)=\supp(\nu_0\mid T_N(i,\gamma)).$$
Let $N_2=[\frac{\beta N}{\gamma p_0}]+1$, $\epsilon= \max\{\delta
(p_1+\gamma+p_0),\,\gamma\}$.  Then
$$T_N(1,\gamma)\subset \Lambda_{N_2}(\beta,\beta, \epsilon)
\quad \mbox{and}\quad T_N(1,\gamma)\subset
\widetilde{\Lambda}_{N_2}(\beta,\beta, \epsilon).$$
 For any $x\in
\Gamma_N(1,\gamma)$, $z\in \pi^{-1}(x)$, it holds that
$$d(f^i(z),f^i_0(x))=d(f^i(z),\, \pi (f^i(x)))<Cr$$ which implies
that if $f^i_0(x)\notin R_1$ then $f^i(z)\notin B(O,r)$ because
$d(B(O,r),\partial R_1)>Cr$. Thus
$$\pi^{-1}(\Gamma_N(1,\gamma))\subset T_N(1,\gamma)\subset
\Lambda_{N_2}(\beta, \beta, \epsilon)$$ which yields that
$$\pi^{-1}(\widetilde{\Gamma}_N(1,\gamma))\subset
\widetilde{\Lambda}_{N_2}(\beta, \beta, \epsilon).$$ For any ergodic
$\nu\in U$, $\pi_{*}\nu\in U_1$. So
$\pi_{*}\nu(\widetilde{\Gamma}_{N}(1,\gamma))>0$. We obtain
$$\nu(\widetilde{\Lambda}_{N_2}(\beta, \beta, \epsilon))
\geq
\nu(\pi^{-1}(\widetilde{\Gamma}_{N}(1,\gamma)))=\pi_{*}\nu(\widetilde{\Gamma}_{N}(1,\gamma))>0.$$
Once more, the ergodicity of $\nu_0$ concludes
$\nu(\widetilde{\Lambda}(\lambda_1, \lambda_1, \epsilon))=1$.

\end{proof}

\noindent{\it Acknowledgement. } The authors thank very much to the
whole seminar of dynamical systems in Peking University. The
manuscript  was improved according to  their many  suggestions.

\end{document}